\documentclass[a4paper,12pt,reqno]{amsart}

\usepackage{amssymb}
\usepackage{mathtools} 

\usepackage{amsfonts}
\usepackage[english]{babel}
\usepackage[utf8]{inputenc}
\usepackage{bbold}
\usepackage{calrsfs}
\usepackage{graphicx} 

\usepackage[dvipsnames]{xcolor} 
\usepackage{fullpage}
\usepackage{hyperref}

\usepackage{scalerel,stackengine} 

\usepackage{enumitem} 
	\setlist[itemize]{left= \parindent .. 2\parindent, label=\raisebox{0.14ex}{\scriptsize$\bullet$}} 

\newtheorem{theorem}{Theorem}[section]
\newtheorem{lemma}[theorem]{Lemma}
\newtheorem{prop}[theorem]{Proposition}
\newtheorem{cor}[theorem]{Corollary}

\theoremstyle{definition}

\theoremstyle{remark}
\newtheorem{remark}[theorem]{Remark}

\numberwithin{equation}{section}

\newcommand{\la}{\label}

\newcommand{\Res}{\mathtt{Res}}

\newcommand{\Ind}{\mathrm{Ind}}
\newcommand{\Mat}{\mathrm{Mat}}

\newcommand{\End}{\mathtt{End}}

\newcommand{\id}{\mathrm{id}}

\def\c{\mathbb{C}}

\def\Z{\mathbb{Z}}

\def\N{\mathbb{N}}

\def\A{\mathcal{A}}
\def\H{\mathcal{H}}

\def\rH{\widehat{\mathsf H}}

\def\Hsf{\mathsf{H}}
\def\Asf{\mathsf A}

\def\L{\mathsf L}

\def\cc{\mathrm{c}}
\def\pp{{\hat p}}

\def\Vc{V^{\vee}}
\def\Sc{W^\vee}

\def\D{\mathcal{D}}

\def\Wh{\widehat W}
\def\hw{\widehat {w}}
\def\hs{\widehat {s}}
\def\Ah{A_\hbar}

\def\BB{\mathtt B}

\def\g{g}

\title{\textit{R}-matrix Dunkl operators and \\ spin Calogero--Moser system}

\author{Oleg Chalykh}
\address{School of Mathematics, University of Leeds, Leeds LS2 9JT, UK}
\email{o.chalykh@leeds.ac.uk}
\author{Maria Matushko}
\address{Steklov Mathematical Institute of Russian
Academy of Sciences, Gubkina str. 8, 119991, Moscow, Russia}
\email{matushko@mi-ras.ru}

\begin{document}

\begin{abstract}
We construct a quantum integrable model which is an $R$-matrix generalization of the Calogero-Moser system, based on the Baxter--Belavin elliptic $R$-matrix. This is achieved by introducing $R$-matrix Dunkl operators so that commuting quantum spin Hamiltonians can be obtained from symmetric combinations of those. We construct quantum and classical $R$-matrix Lax pairs for these systems. In particular, we recover in a conceptual way the classical $R$-matrix Lax pair of Levin, Olshanetsky, and Zotov, as well as the quantum Lax pair found by Grekov and Zotov. Finally, using the freezing procedure, we construct commuting conserved charges for the associated quantum spin chain proposed by Sechin and Zotov, and introduce its integrable deformation. Our results remain valid when the Baxter--Belavin $R$-matrix is replaced by any of the trigonometric $R$-matrices found by Schedler and Polishchuk in their study of the associative Yang--Baxter equation.  
\end{abstract}

\maketitle

\section{Introduction}
The purpose of the present paper is to study an $R$-matrix generalisation of the quantum Calogero--Moser system. The (usual, scalar) Calogero--Moser system \cite{Ca71} is a celebrated example of a completely integrable system, with numerous connections to several areas of mathematics and physics. It describes a system of $n$ particles on the line, with positions denoted as $x_1, \dots, x_n$, whose pairwise interaction is governed, in the rational case, by the potential $u_{ij}=(x_i-x_j)^{-2}$; there are also trigonometric and elliptic versions, with $u_{ij}=\sin^{-2}(x_i-x_j)$ or $\wp(x_i-x_j)$, respectively. Adding spin interaction between the particles leads to the so-called spin Calogero--Moser systems \cite{GH, Woj85, HaHaldane92, Poly92,  HikamiWadati92, MP,BGHP,Hal}. 

The generalisation we are going to consider describes a version of the quantum spin Calogero--Moser system in which spin interaction is governed by an $R$-matrix, more specifically, the elliptic Baxter--Belavin $R$-matrix \cite{Baxter,Belavin}.
It originates in the work of Levin, Olshanetsky and Zotov \cite{LOZ} in which 
a classical Lax pair was found with such $R$-matrices appearing as matrix entries, and with the Lax equation describing the particle dynamics of the usual (scalar) Calogero--Moser system. In \cite{GSZ}, this was extended to a classical integrable model of interacting spinning tops. Additionally, a quantum analogue of the Lax pair \cite{LOZ} was found in \cite{GZ} and a quadratic $R$-matrix quantum Hamiltonian proposed, however, the question about its integrability remained open. In \cite{SZ}, Sechin and Zotov used the Lax pair \cite{LOZ} and some empirical method to propose a quantum long-range spin chain, and some further checks indicated its integrability. 
 All this suggested that a quantum-mechanical integrable system should exist which would be behind the $R$-matrix Lax pairs in \cite{LOZ},\cite{GZ} and the quantum spin chain in \cite{SZ}, and our goal is to construct such a system. 

 Our main idea is very simple: since the Calogero--Moser systems are best understood with the help of Dunkl operators \cite{Du}, we need to upgrade these operators to incorporate $R$-matrices. As it turns out, this idea is not new: in \cite{FoKi} (cf. \cite{K}), Fomin and Kirillov, in a more abstract setting motivated by Schubert calculus, looked for an associative algebra $\mathcal R$ with an action of the symmetric group $S_n$, admitting a commutative family of elements $\theta_i=\sum_{j (\ne i)}^n \! r_{ij} {s}_{ij}\in \mathcal R*S_n$, $i=1, \dots, n$\footnote{For another fruitful approach to incorporating $r$-matrices into the ($q$)KZ systems and Dunkl operators, see \cite{FR, Ch2}.}. Assuming $r_{ij}=-r_{ji}$ and covariance of $r_{ij}$, that is, $\sigma r_{ij}\sigma^{-1}=r_{\sigma(i)\sigma(j)}$ for any $\sigma\in S_n$, one easily finds that the elements $r_{ij}\in \mathcal R$ must satisfy the following quadratic relations: 
\begin{equation}\label{AYBintro}
r_{ij}r_{jk}+r_{ik}r_{ji}+r_{jk}r_{ki}=0.
\end{equation}
A particular solution in rational functions, $r_{ij}=g(x_i-x_j)^{-1}$, leads to commuting rational Dunkl operators $y_i:=\partial/\partial x_i-\theta_i$. Another possibility is to consider a single $r$-matrix $r\in\End(\c^d\otimes\c^d)$ and let $r_{ij}$ act on the $n$-fold tensor product $(\c^d)^{\otimes n}$ by applying $r$ to the $i$th and $j$th tensor factors. In that case, the relations \eqref{AYBintro} boil down to a single equation for $r$. This equation also appeared in the context of infinitesimal Hopf algebras \cite{A1, A2}. It became known as the (constant) \emph{associative Yang-Baxter equation} (AYBE) after Polishchuk who, in \cite{Pol1}, introduced the following more general version with spectral parameters: 
\begin{equation} \label{eq:AYBE}
    R_{12}(z,\mu) \, R_{23}(z',\mu') = R_{13}(z+z',\mu') \, R_{12}(z,\mu-\mu') + R_{23}(z',\mu'-\mu) \, R_{13}(z+z',\mu) \, .
\end{equation} 
Polishchuk observed its close relationship with both the classical and the quantum Yang--Baxter equations. This led to the classification of its elliptic \cite{Pol1} as well as trigonometric solutions \cite{S, Pol2}, under some additional constraints, with the most important solution given by the Baxter--Belavin $R$-matrix.  

On the other hand, in the scalar case $d=1$, the equation \eqref{eq:AYBE} coincides with the functional equation studied in \cite{BFV} as the condition of commutativity of elliptic Dunkl operators. In that case its general solution is given by the Kronecker function or its degenerations, see \eqref{Krphi} below and \cite{BFV, Pol2} for further details. Thus, all that is needed is replacing the Kronecker's function in the formula for the elliptic Dunkl operators from \cite{BFV} with a solution of the AYBE equation. This produces a family of matrix-valued Dunkl operators whose commutativity holds precisely due to \eqref{eq:AYBE}. 

Once we have such $R$-matrix Dunkl operators $y_i$ in place, we can use constructions known for the scalar case \cite{EFMV, C19, C24}, suitably modified. Thus, $n$ commuting quantum Hamiltonians are obtained from symmetric combinations of $y_i$ similarly to \cite{EFMV, C19}, quantum and classical Lax pairs are derived similarly to \cite{C19}, and the freezing procedure leading to a long-range quantum spin chain is carried out similarly to \cite{C24}. In particular, we recover the Lax pairs from \cite{LOZ},\cite{GZ} in a conceptual way (as well as obtain Lax pairs for the higher-order Hamiltonians), and construct conserved charges for the quantum spin chain from \cite{SZ}. Furthermore, in parallel with \cite{C24}, we construct an integrable deformation of that spin chain. Note that the conventional quantum spin Calogero--Moser system and the Inozemtsev spin chain \cite{I1}, considered in \cite{C24}, can be viewed as a particular case of our constructions, for a special choice of an $R$-matrix (see Remark \ref{rino}). Our constructions also work for the trigonometric solutions to the AYBE found in \cite{S, Pol2}.      

We should mention the $R$-matrix spin Ruijsenaars system and the related quantum spin chain recently constructed in \cite{MZ1, MZ}. Those constructions rely on explicit formulas for the Hamiltonians (inspired by the results of \cite{L18, LPS}) and certain non-trivial identities for the Baxter--Belavin $R$-matrix (generalising the elliptic identities from \cite{R}). The integrable system and the spin chain that we consider can be viewed as a non-relativistic $q\to 1$ limit of those in \cite{MZ1, MZ}. However, in the non-relativistic case no compact explicit general formulas for the Hamiltonians are known, so deriving our results from those in \cite{MZ1, MZ} in a rigorous fashion would be a complicated and ungratifying task. Besides, this would not recover the integrable deformation of the Sechin--Zotov spin chain that we found. Moreover, the approach presented here can be extended to the relativistic case, providing a more illuminating derivation of the models \cite{MZ1, MZ} that avoids complicated algebraic manipulations. This will be the subject of the forthcoming work \cite{CL}.    

Another interesting family of $R$-matrix spin Ruijsenaars models has been proposed recently in \cite{KL, KL1, KL2} based on Felder's dynamical $R$-matrix \cite{Fe1, Fe2}; see also \cite{ReSt22, ReSt20, Re23} where families of hyperbolic spin Calogero--Moser systems are constructed and studied, with Felder's trigonometric dynamical $r$-matrices making an appearance. Our constructions do not work in that case since Felder's $R$-matrix does not satisfy the AYBE. It would be interesting to extend our approach to allow more general $R$-matrices. 

Among other possibilities, let us mention Calogero--Moser systems for other root systems, see \cite{OP1,OP2} and \cite{Ch1} for the spin case. In that direction, a $BC_n$-version of the Lax pair from \cite{LOZ} has been found recently in \cite{MMZ}. It involves, beyond the Baxter--Belavin $R$-matrix, a suitable boundary matrix $K$ satisfying an associative analogue of the reflection equation. In the forthcoming work \cite{CM} we will extend the present approach to the $BC_n$ case, in particular, providing an alternative derivation of the results of \cite{MMZ} and a construction of the related spin chain.    

\medskip

The structure of the paper is as follows. In Section 2 we set up the notation, list the requirements for the $R$-matrix and describe the Baxter--Belavin matrix and its properties. In Section 3 we introduce $R$-matrix Dunkl operators and then use them in Section 4 to construct the $R$-matrix Calogero--Moser Hamiltonians. The central result here is Proposition \ref{main} whose proof is given in Section 5. In Section 6 we derive quantum and classical Lax pairs and calculate them explicitly for the quadratic Hamiltonian, recovering the classical and quantum Lax pairs from \cite{LOZ}, \cite{GZ}. In Section 7 we apply the freezing procedure to construct commuting charges for the Sechin--Zotov quantum spin chain, and in Section 8 we describe an integrable deformation of that spin chain. Finally, in Section 9 we explain how to adjust our constructions for the case of trigonometric solutions to the AYBE. 

\medskip

\noindent\textbf{Acknowledgement.} We would like to thank Anatoli Kirillov, Rob Klabbers, Jules Lamers, Andrii Liashyk, Nicolai Reshetikhin, Ivan Sechin, and Andrei Zotov for useful comments and discussions. We are especially grateful to Jules Lamers for his careful reading of an earlier draft and for suggesting numerous improvements to the presentation. This collaboration started when both authors visited Beijing Institute of Mathematical Sciences and Applications (BIMSA). We thank BIMSA for the hospitality and excellent working conditions.\\ The work by M.M. is supported by the Ministry of Science and Higher Education of the Russian Federation (agreement no. 075-15-2025-303) and by the Foundation for the Advancement of Theoretical Physics and Mathematics “BASIS”.
\section{Preliminaries}\label{sec: Notation and setup}

\subsection{}
Consider $n$ particles moving on a circle and carrying a spin. 
We will work over complex numbers, so all vector spaces, linear maps and tensor products are over $\c$.
Let $W=S_n$ be the symmetric group in its reflection representation $V=\c^n$.
We view $V$ as a complex Euclidean vector space with 
the standard orthonormal basis $(e_i)$ and coordinates $(x_i)$. {We think of the $x_i$ as the coordinates of the particles.} The symmetric group acts on $V$ by permuting the basis vectors. This induces a natural $W$-action on
the space $\c(V)$ of meromorphic functions on $V$ by $(w.f)(v)=f(w^{-1} . v)$ for $v\in V$, $w\in W$, and on the ring $\D(V)$ of differential operators on $V$ with meromorphic coefficients, with $w.\partial_\xi=\partial_{w.\xi}$ for $\xi\in V$, $w\in W$. In addition, consider the space $U = U_1\otimes\dots\otimes U_n$ with each $U_i\cong\c^d$ and with $W$ acting by permuting the tensor factors in $U$. The space $U$ models the spins of the $n$ particles.

We write $\D_U(V) \coloneqq \D(V)\otimes\End \, U$ for the ring of matrix{(-valued)} differential operators on $V$. Consider the `diagonal' action of $W$ on $\D_U(V)$, with $w\in W$ acting simultaneously on $\D_V$ and (by the adjoint action) on $\End\, U$.
We use a hat, $\widehat{\,\cdot\,}$, to distinguish it from other actions:
\begin{equation}\label{act}
    \hw . (a\otimes b) = w . a\otimes w \, b \, {w^{-1}}\,,
    \qquad a\in\D(V) \, ,\  b\in\End\, U \, , \quad w\in W \, .
\end{equation}
This defines the group $\Wh\simeq W$ acting by automorphisms of $\D_U(V)$, so we can
form the crossed product $\D_U(V)*\Wh$. As an algebra, $\D_U(V)*\Wh$ is generated by its two subalgebras $1\otimes\c \Wh$ and $\D_U(V)\otimes 1$, which can be identified with {the group algebra} $\c \Wh$ and $\D_U(V)$, respectively. Using these identifications, we simply write $a \, \hw$ for $a\otimes \hw$, so that each element of $\D_U(V)*\Wh$ is written uniquely as $a=\sum_{w\in W}a_w \, \hw$ with coefficients $a_w\in\D_U(V)$, and with the product described by
\begin{equation*}
    a \, b = \sum_{w, w'\in W} \!\! a_w \,\hw \; b_{w'}\,\hw'=\sum_{w, w'\in W} \! \bigl(a_w \, (\hw . b_{w'})\bigr)(\hw\, \hw')\,.
\end{equation*}
The action \eqref{act} extends to $\D_U(V)*\Wh$ by 
\begin{equation}\label{act1}
 \hw . (a\,\hw') =(\hw . a)(\hw \, \hw' \hw^{-1})\,,\qquad a\in\D_U(V)\,, \quad w,w'\in W.   
\end{equation}

\subsection{}
Let $e_{\alpha\beta} \in \End\, \c^d$, $1\leqslant \alpha,\beta \leqslant d$, denote the matrix units, so that $P = \sum_{\alpha,\beta=1}^d e_{\alpha\beta} \otimes e_{\beta\alpha}$ is the permutation operator on $\c^d \otimes \c^d$. Then the transposition $(ij)\in W$ acts on $U$ by $P_{ij} = \sum_{\alpha,\beta=1}^d 1^{\otimes (i-1)} \otimes e_{\alpha\beta} \otimes \!1^{\otimes (j-i-1)} \otimes e_{\beta\alpha} \otimes \!1^{\otimes (n-j)}$.

Consider an $R$-matrix $R\in \End(\c^d\otimes \c^d)$ depending on {spectral parameters} $z,\mu\in\c$ so that $R=R(z,\mu)$ is a meromorphic function $\c^2\to \End(\c^d\otimes \c^d)$, $(z,\mu)\mapsto R(z,\mu)$. {In our main example, $\mu$ plays the role of `crossing parameter'.}
We employ the standard `tensor-leg' notation $R_{ij}(z,\mu)$ for the endomorphism of $U$ that acts as $R(z,\mu)$ on $U_i\otimes U_j$ and as the identity on the other tensor factors of $U$. We have $R_{ji}(z,\mu) = P_{ij} \, R_{ij}(z,\mu) \, P_{ij}$.
We make the following assumptions on the \textit{R}-matrix.

\begin{enumerate}[label=(\roman*)]
    \item {\emph{Skew symmetry}:} $R(-z,-\mu) = -P \, R(z,\mu) \, P$. \label{it:skew}
     \item {\emph{Associative Yang--Baxter equation (AYBE) \eqref{eq:AYBE}}}\label{it:aybe}.%
    \item {\emph{Unitarity}:} \label{it:unitarity} $R(z,\mu) \, R_{21}(-z,\mu) = \bigl(\wp(\mu)-\wp(z)\bigr) \, \mathrm{Id}$, with $\wp$ the Weierstra{\ss} elliptic function or any of its degenerations.
    \item {\emph{Regularity}:} \label{it:regular} As a function of $\mu$, $R(z,\mu)$ has a first order pole at $\mu=0$, with residue $\mathrm{res}_{\mu=0} R(z,\mu)$ independent of $z$.
       \end{enumerate} 
The next property is a consequence of  \ref{it:skew}--\ref{it:unitarity}, see \cite[Sec.\ 4]{LOZ} and \cite[Theorem~1.4]{Pol2}, and explains the terminology ``$R$-matrix'':
       \begin{itemize}
\item[] \emph{Quantum Yang--Baxter equation (QYBE):} \label{it:ybe}
\begin{equation} \label{eq:YBE}
    R_{12}(z,\mu) \, R_{13}(z+z',\mu) \, R_{23}(z',\mu) = R_{23}(z',\mu) \, R_{13}(z+z',\mu) \, R_{12}(z,\mu) \, .
\end{equation}
    \end{itemize}
Note that, unlike \eqref{eq:YBE}, the relation \eqref{eq:AYBE} is highly sensitive to the normalisation of the $R$-matrix. This is the reason for the precise factor in \ref{it:unitarity}. 

Introduce additional ``spectral'' variables $\lambda=(\lambda_1, \dots, \lambda_n)$. For $i\ne j$ and $k\ne l$, we abbreviate $R_{ij}^{kl} \coloneqq R_{ij}(x_i - x_j,\lambda_k - \lambda_l)$ and further $R_{ij} \coloneqq R_{ij}^{ij}$. In this notation, we have 

 \begin{itemize}
\item[] \emph{Covariance property:} $\widehat{w} \, R^{kl}_{ij} \, \widehat{w}^{-1} = R_{w(i) \, w(j)}^{kl}$ for any $w\in W$. 
 \end{itemize}
%

%
Then \ref{it:skew}--\ref{it:aybe} imply that for $i\ne j$, $k\ne l$ and pairwise distinct $p,q,r$ we have
\begin{equation}\label{eq:AYBEi}
R^{kl}_{ij}=-R^{lk}_{ji} \,, \quad R_{pq}^{pq} \, R_{qr}^{pr} + R_{pr}^{pr} \, R_{qp}^{qr} + R_{qr}^{qr} \, R_{rp}^{qp}=0\,.
\end{equation}
It also easily follows from the definitions that 
\begin{equation}\label{eq:Rcommute}
[R^{kl}_{ij},R^{rs}_{pq}]=0\quad\text{if}\quad \{i,j\}\cap\{p,q\}=\varnothing\,.
\end{equation}

\subsection{}
For a simple example of such $R$, one can take $U$ the trivial representation (i.e.\ $d=1$) and $R$ equal to the Kronecker function: 
\begin{equation}\label{Krphi}
 R(z,\mu)=\phi(z,\mu)=\phi(z,\mu|\tau) \coloneqq \frac{\theta'(0) \, \theta(z+\mu)}{\theta(z)\,\theta(\mu)}\,,
\end{equation}
where
%
$\theta=\theta(z|\tau)$
is the first (odd) Jacobi theta-function.
The Kronecker function has all the properties \ref{it:skew}--\ref{it:regular}, where \ref{it:aybe} is nothing but Fay's trisecant identity, see \cite{Weil}. 
%
%
%

In case $d>1$ one can take, according to \cite{Pol1, LOZ}, the Baxter--Belavin elliptic $R$-matrix \cite{Belavin} or its degenerations. For $d=2$ this is Baxter's eight-vertex $R$-matrix \cite{Baxter}:
 \begin{equation*}\label{BaxterR}
 \begin{array}{c}
\displaystyle{
R(z,\mu)=\frac{1}{2}
}
\left(
 \begin{array}{cccc}
 \varphi_{00}+\varphi_{10} & 0 & 0 & \varphi_{01}-\varphi_{11}
 \\
 0 & \varphi_{00}-\varphi_{10} & \varphi_{01}+\varphi_{11} & 0
  \\
 0 & \varphi_{01}+\varphi_{11} & \varphi_{00}-\varphi_{10} & 0
 \\
 \varphi_{01}-\varphi_{11} & 0 & 0 & \varphi_{00}+\varphi_{10}
 \end{array}
 \right)\,,
  \end{array}
 \end{equation*}
 where
 $
 \varphi_{00}=\phi(z,\frac{\mu}{2})\,,\,
 \varphi_{10}=\phi(z,\frac{1}{2}+\frac{\mu}{2})\,,\,
 \varphi_{01}=e^{\pi\mathrm{i} z}\phi(z,\frac{\tau}{2}+\frac{\mu}{2})\,,\,
 \varphi_{11}=e^{\pi\mathrm{i} z}\phi(z,\frac{1+\tau}{2}+\frac{\mu}{2})\,.
 $

\medskip
 
In the general case, the Baxter--Belavin $R$-matrix is defined as
%
%
 \begin{equation}\label{bb1}
 R(z,\mu)= \sum_{\genfrac{}{}{0pt}{}{\alpha,\beta,\gamma,\delta=1,\dots, d}{\alpha+\gamma\equiv\beta+\delta\ (\mathrm{mod}~d)}} \!\!\!\!\!\!\!\!\!\! R_{\alpha\beta,\gamma\delta}(z,\mu)\, e_{\alpha\beta}\otimes e_{\gamma\delta}\,,
 \end{equation}
where
 \begin{align}
 R_{\alpha\beta,\gamma\delta}(z,\mu)=&\exp\Bigl(\frac{2\pi \mathrm{i}}{d} \bigl((\gamma-\beta)(\beta-\alpha)\,\tau+(\gamma-\beta)\mu+(\beta-\alpha)\,z \bigr)\Bigr)\nonumber\\\label{bb2}
 &\times\phi(z+(\gamma-\beta)\,\tau,\mu+(\beta-\alpha)\,\tau\,|\,d\,\tau)\,.
 \end{align}

\subsection{}
The Baxter--Belavin $R$-matrix has the following 
symmetries:
\begin{align}
R(z,\mu)&=R(\mu,z)\,P\,, \label{eq:R_symmetry}\\
[Q\otimes Q\,, R(z,\mu)]&=[\Lambda\otimes \Lambda\,, R(z,\mu)]=0\,.
\end{align}
Its behaviour under the shifts by the periods $1,\tau$ is as follows:
 \begin{gather}\label{trans}
 R(z,\mu+1)=(Q^{-1}\otimes \mathrm{Id})\,R(z,\mu)\,(\mathrm{Id}\otimes  Q)\,,
 \\
 R(z,\mu+\tau)=\exp\left(-\frac{2\pi\mathrm{i}z}{d}\right)\,(\Lambda^{-1}\otimes \mathrm{Id})\,R(z,\mu)\,(
 \mathrm{Id}\otimes\Lambda)\,,
 \end{gather}
where $Q,\Lambda\in \End\, \c^d$ with 
\begin{equation}\label{QL}
 \displaystyle{
 Q_{kl}=\delta_{kl}\exp\left(\frac{2\pi
 \mathrm{i}k}{{ d}}\right)\,,\quad
 \Lambda_{kl}=\delta_{k-l+1\equiv 0\,({\hbox{\tiny{mod}}}\, d)}}\,
,\qquad k,l=1,\dots, d.   
\end{equation}
Similar periodicity properties in the first argument of the $R$-matrix follow from \eqref{eq:R_symmetry}. 
Also, $R$ admits the following expansion at $\mu=0$:
 \begin{equation}\label{expR}
 R(z,\mu)=\frac{1}{\mu} \, \mathrm{Id}+r(z)+\mu\, m(z)+O(\mu^2)\,.
 \end{equation}
Here $r(z)\in \End(\c^d\otimes \c^d)$ is the classical Belavin--Drinfeld elliptic $r$-matrix \cite{BD} %
 %
 %
 which satisfies the classical Yang--Baxter equation (CYBE):
 \begin{equation}\label{YBcl}
 \begin{array}{c}
    \displaystyle{
    [r_{12},r_{23}]+[r_{12},r_{13}]+[r_{13},r_{23}]=0\,,
    \qquad r_{ij}=r_{ij}(z_i-z_j)\,.
    }
\end{array}
\end{equation}
The residue of $R(z,\mu)$ (and of
$r(z)$) at $z=0$ is the permutation operator $P$.
%
%
The skew-symmetry of $R$ implies
 \begin{equation}\label{r056}
 \begin{array}{c}
    \displaystyle{
    r(z)=-r_{21}(-z)\,,\qquad
    m(z)=m_{21}(-z)\,,\qquad
    r'(z)=r'_{21}(-z)\,,
    }
\end{array}
\end{equation}
where $r'(z)=\partial_z \, r(z)$.
From the unitarity property we find $m(z)=\frac{1}{2} \bigl(r^2(z)- \wp(z) \, \mathrm{Id} \bigr) $.

 %

\begin{remark}
The AYBE equation \eqref{eq:AYBE} was introduced by Polishchuk \cite{Pol1} in relation with triple Massey products in certain $A_\infty$-categories of geometric origin. He also found all of its nondegenerate elliptic solutions, which include the Baxter--Belavin matrix \eqref{bb1}--\eqref{bb2}, see \cite[Sec.~2]{Pol1}. A classification of nondegenerate trigonometric solutions was obtained in \cite{S, Pol2}. The important role of the associative Yang--Baxter equation in the theory of integrable systems was emphasised and explored in \cite{LOZ, K, KZ, SZ, GSZ, MZ1, MZ2}.

\end{remark}


\remark\label{rino}{For our main example, i.e.\ the Baxter--Belavin $R$-matrix, the regularity property is complemented by the fact that the residue of $R$ at $\mu=0$ is the identity matrix, see \eqref{expR}. This will not be important for our constructions, but will only affect the explicit form of the resulting Hamiltonians. For a different example, one can choose $R(z,\mu) = \phi(z,\mu) \, P$, which satisfies properties \ref{it:skew}-\ref{it:unitarity}, but have the permutation matrix $P$ as its residue at both $\mu=0$ and $z=0$. Such a choice leads to the spin Calogero--Moser system and the Inozemtsev spin chain, treated in \cite{C24}.}

\remark{Our results remain valid for all non-degenerate trigonometric solutions of the associative Yang--Baxter equation, classified in \cite{S,Pol2}, as well as for the supersymmetric $R$-matrices from \cite{MZ2}, see Sec.~\ref{trigcase}}.



\section{{\it R}-matrix Dunkl operators}

\noindent
\subsection{}
Let $R(z,\mu)$ be an $R$-matrix satisfying the assumptions of Section \ref{sec: Notation and setup}. Given a coupling parameter $\g \in \c$, we define the (quantum)
\emph{$R$-matrix Dunkl operators} $y_1, \dots, y_n$
as
\begin{equation}\label{RD}
    y_i=\hat p_i  - \g\sum_{j (\ne i)}^n \! R_{ij} \, \widehat{s}_{ij} \,, \qquad  \, \hat p_i \coloneqq \hbar \, \frac{\partial}{\partial x_i} \, . 
\end{equation}
To emphasise that these are $\lambda$-dependent elements of
$\D_U(V)* \Wh$
we may write $y_i=y_i(\lambda)$.
For $\xi\in V$, we set $y_\xi = \sum_{i=1}^n \xi_i \, y_i$. The main properties of these Dunkl operators are given in the following lemma.

\begin{lemma}\label{lemma1}
The elements $y_\xi=y_\xi(\lambda)$ of $\D_U(V)* \Wh$ have the following properties.
\begin{enumerate}
    \item Commutativity: $[y_\xi, y_{\eta}]=0$ for any $\xi, \eta\in V$.
    \item Equivariance: $\hw \, y_\xi(\lambda) \, \hw^{-1} = y_{w{.}\xi}(w {.} \lambda)$.
\end{enumerate}
\end{lemma}
\begin{proof} Equivariance is obvious.
    For the commutativity, we need to check that $[y_i, y_k]=0$ for $i\neq k$. The coefficient at $\hbar\,{\g}$ vanishes due to
    \begin{align*}
   \biggl[\frac{\partial}{\partial x_i},\sum_{l (\ne k)}^n \! R_{kl} \, \widehat{s}_{kl} \biggr] + \biggl[\sum_{j (\ne i)}^n \! R_{ij} \, \widehat{s}_{ij},\frac{\partial}{\partial x_k}\biggr] = \biggl[\frac{\partial}{\partial x_i},R_{ki} \, \widehat{s}_{ik}\biggr] + \biggl[R_{ik} \, \widehat{s}_{ik},\frac{\partial}{\partial x_k}\biggr] =0.
    \end{align*}
To calculate the coefficient at ${\g}^2$, we use that for $i\neq j,\, k\neq l,\,i\neq k$
\begin{equation*}
 [ R_{ij} \, \widehat{s}_{ij},  R_{kl} \, \widehat{s}_{kl}]  =\begin{cases}
0&\text{for}\ \{i,j\}\cap\{k,l\}=\varnothing \,,\\
0&\text{for}\ j=k,\, l=i\,,\\
R_{ik} \, R_{il}^{kl} \, \widehat{s}_{ik} \, \widehat{s}_{kl}-R_{kl} \, R_{il}^{ik} \, \widehat{s}_{kl} \, \widehat{s}_{ik}&\text{for}\ j=k,\, l\neq i \,,\\
R_{ij} \, R_{kj}^{ki} \, \widehat{s}_{ij} \, \widehat{s}_{ik}-R_{ki}\,R_{kj}^{ij}\,\widehat{s}_{ik}\,\widehat{s}_{ij}&\text{for}\ j\neq k,\, l=i \,,\\
R_{ij} \, R_{ki}^{kj} \, \widehat{s}_{ij}\,\widehat{s}_{jk}-R_{kj}\,R_{ik}^{ij} \, \widehat{s}_{jk} \, \widehat{s}_{ij}&\text{for}\ j\neq k,\, l=j ,
\end{cases}
  \end{equation*}
together with the relations $\widehat{s}_{ij} \, \widehat{s}_{ik}=\widehat{s}_{jk} \, \widehat{s}_{ij}=\widehat{s}_{ik} \, \widehat{s}_{jk}$ in the symmetric group $W$. This gives
  \begin{align*}
  \biggl[\sum_{j (\ne i)}^n \! R_{ij} \, \widehat{s}_{ij},\sum_{l (\ne k)}^n \! R_{kl} \, \widehat{s}_{kl}\biggr]=&\sum_{j (\ne i,k)}^n(R_{ik}{R_{ij}^{kj}}  +R_{ij}R_{kj}^{ki}-R_{kj}R_{ik}^{ij} ) \, \widehat{s}_{ij}\,\widehat{s}_{ik}
  \\&
    + \!\sum_{j (\ne i,k)}^n(R_{ij}R_{ki}^{kj}-R_{ki}R_{kj}^{ij}-R_{kj}R_{ij}^{ik}) \, \widehat{s}_{ij} \, \widehat{s}_{jk}=0\,.
  \end{align*}
Both coefficients, at $\widehat{s}_{ij}\widehat{s}_{ik}$ and $\widehat{s}_{ij}\widehat{s}_{jk}$, vanish due to the associative Yang--Baxter equation~\eqref{eq:AYBEi} and the skew-symmetry property~\ref{it:skew} of the $R$-matrix.
  \end{proof}
  \begin{remark}
The scalar choice \eqref{Krphi} corresponds to the elliptic Dunkl operators introduced by Buchstaber, Felder and Veselov \cite{BFV}.
  \end{remark}

\begin{remark}
    The idea that the commutativity of Dunkl operators of a general form \eqref{RD} is related to the associative Yang--Baxter equation goes back to \cite{FoKi}, cf. \cite{K}. Our setup is only slightly different as we allow $R_{ij}$ to depend on additional spectral variables. 
\end{remark}

\subsection{}
The operators $y_i$ have classical counterparts that are obtained by taking $\hbar\to 0$. To formalise this procedure, we use the framework of formal deformations, following the setup in \cite[Sec.~2.3]{C19}, see also \cite[Sec.~3.1]{E} for the general framework and references. We replace the ring $\D(V)$ of differential operators with the algebra 
\begin{equation}\label{ah}
\Ah=\c[[\hbar]]\otimes \c(V)[\hat p_1, \dots, \hat p_n]\,,\qquad \hat p_i=\hbar \, \partial_{i}\,.
\end{equation}
The quantum momenta $\hat p_i$ satisfy the relations $[\hat p_i, f]=\hbar\,\partial_{i} f \coloneqq \hbar\otimes \partial_{i}f$ for $f\in\c(V)$.
We have a linear map
\begin{equation}\label{eta}
\eta_0 \colon \Ah\to A_0\,,\qquad f\mapsto f\,,\ \ \hat p_i\mapsto p_i\,,\ \ \hbar\mapsto 0\,,
\end{equation}
where 
\begin{equation}\label{a0}
A_0=\c(V)\otimes\c[V^*]=\c(V)[p_1, \dots, p_n]    
\end{equation}
is the classical, commutative analogue of $\Ah$.
The map $\eta_0$ induces an algebra isomorphism $\eta_0 \colon \Ah/\hbar \, \Ah\overset{\sim}{\to} A_0$.
Therefore, $\Ah$ is a formal deformation of $A_0$. This equips $A_0$ with the Poisson bracket defined by $\{\eta_0(a), \eta_0(b)\}=\eta_0(\hbar^{-1}[a,b])$ for $a,b\in \Ah$. In particular, $\{p_i,x_j\}=\delta_{ij}$.
The map $\eta_0$ is compatible with the action of $W$, so it can be naturally extended to a linear map
\begin{equation*}
\eta_0 \colon (\Ah\otimes\End\,U)*\Wh \to (A_0\otimes\End\,U)*\Wh\,,   
\end{equation*}
and an algebra isomorphism 
\begin{equation*}
\eta_0 \colon \bigl((A_\hbar/\hbar \, A_\hbar)\otimes\End\,U\bigr)*\Wh \overset{\sim}\to (A_0\otimes\End\,U)*\Wh\,.   
\end{equation*}
Hence, $(\Ah\otimes\End\,U)*\Wh$ is a formal deformation of $(A_0\otimes\End\,U)*\Wh$.
For any element $a\in(\Ah\otimes\End\,U)*\Wh$, we call $\eta_0(a)$ the \emph{classical limit} of $a$.

\medskip

We can now view the $R$-matrix Dunkl operators $y_i$ as elements of $(\Ah\otimes\End\,U)*\Wh$, so their classical limit $y_i^{\cc} \coloneqq \eta_0(y_i)$ is 
\begin{equation}\label{RDc}
y_{i}^{\cc}=p_i-g\sum_{j(\ne i)}^n R_{ij} \, \widehat{s}_{ij}\,,\qquad i=1,\dots, n\,.
\end{equation}
These will be referred to as \emph{classical $R$-matrix Dunkl operators}; they form a commutative, equivariant family of elements in $(A_0\otimes\End\,U)*\Wh$.

\remark{In this setting, the classical limit of $(\Ah\otimes\End\,U)*\Wh$ is a \emph{noncommutative} algebra $(A_0\otimes\End\,U)*\Wh$, so we are not quite in the setting of Poisson deformations, see however \cite{MV} for a Poisson-geometric picture in such situations. There is also a nice framework of hybrid integrable systems developed in \cite{LRS}. It would be interesting to study the integrable systems constructed in the present paper in the frameworks of \cite{MV, LRS}.}

\section{{\it R}-matrix Calogero--Moser Hamiltonians}
\subsection{}
\label{subsection:cecms} We can use $y_i$ to construct an $R$-matrix analogue of the quantum Calogero--Moser system. The idea, cf.~\cite{EFMV}, is to substitute Dunkl operators into suitable \emph{classical} Hamiltonians. Recall that the classical elliptic Calogero--Moser (eCM) system is described by the following Hamiltonian:
\begin{equation}\label{cmc}
h=h(x,p)=\sum_{i=1}^n p_i^2-2\,\g^2\sum_{i<j}^n\wp(x_i-x_j)\,.
\end{equation}
We view $h$ as an element of the algebra \eqref{a0}; note that $h\in A_0^W$ due to its $W$-invariance. 

The eCM system \eqref{cmc} has a Lax presentation \cite{Ca, Kr} and is completely integrable in Liouville sense \cite{Pe}. More concretely, let $\sigma_r(p)$ denote the $r$th elementary symmetric polynomial of $p_1,\dots, p_n$. Then there are $n$ functions $h_r(x,p)\in A_0^W$, $r=1,\dots, n$, such that $h_r=\sigma_r(p)+\text{l.o.t.}$, and $\{h_r, h_s\}=0$ for all $r,s$. Here are a first few of them (see \cite[(4.14), (4.15)]{OP1} for a general formula):
\begin{align*}
    h_1&=p_1+\dots+p_n\,,\\
    h_2&=\sum_{i<j}^n\left(p_i \, p_j + g^2\,\wp(x_i-x_j)\right)\,,\\
    h_3&=\sum_{i<j<k}^n \!\! \left(p_i \, p_j \, p_k+{\g}^2 \, \wp(x_i-x_j) \, p_k+{\g}^2 \, \wp(x_i-x_k) \, p_j+{\g}^2 \, \wp(x_j-x_k) \, p_i\right)\,.
\end{align*}
Note that that the Hamiltonian \eqref{cmc} is $h(x,p)=h_1^2-2\,h_2$. These classical Hamiltonians can be quantised, i.e., lifted to some $H_1, \dots, H_n\in \Ah^W$ with $h_r=\eta_0(H_r)$ and $[H_r, H_s]=0$ for all $r,s$. See \cite{OS} for explicit expressions and \cite{EFMV} for a general construction using elliptic Dunkl operators. In particular, a quantum analogue of $h(x,p)$ is
\begin{equation}\label{cmq}
H=\sum_{i=1}^n \hat p_i^2-2\,\g\,(\g-\hbar)\sum_{i<j}^n\wp(x_i-x_j)\,.
\end{equation}

\subsection{}
Next, we make the substitution 
\begin{equation}\label{subs}
x_i\mapsto \lambda_i,\quad p_i\mapsto y_i     
\end{equation}
in the classical Hamiltonians $h_r(x,p)$. This is well defined since $y_i$ commute between themselves and with all $\lambda_j$. We denote the result symbolically as $h_r(\lambda, y)$; this is a $\lambda$-dependent element of $(\Ah\otimes\End\,U)*\Wh$. Similarly, replacing $p_i$ with the classical Dunkl operators $y_i^{\cc}$ we obtain $h_r(\lambda, y^{\cc})$, an element of $(A_0\otimes\End\,U)*\Wh$.
The following result, whose proof will be postponed to Sec.~\ref{prf}, plays a crucial role and is analogous to \cite[Prop.~5.1]{C19}.  

\begin{prop}\label{main}
$(1)$ For each classical eCM Hamiltonian $h_r(x,p)$, $r=1,\dots, n$, the element $h_r(\lambda, y)$ has a well-defined limit at $\lambda=0$, denoted $h_r(0,y)$. This produces $n$ pairwise commuting elements in $\bigl((\Ah\otimes\End\,U)*\Wh\,\bigr)^W$. 

$(2)$ In the classical case, $h_r(\lambda, y^{\cc})$ does not depend on $\lambda$ and, moreover, reduces to a\,\footnote{\ {We do not claim that it reduces to $h_r(x,p)$ for general $r$, as it may well have lower-order Hamiltonians added to it. In examples we do retrieve $h_r(x,p)$ for $r=2,3$.}}
usual (scalar) classical eCM Hamiltonian. That is, $h_r(\lambda, y^{\cc})=h_r(0, y^{\cc})$ is an element of $\c[h_1,\dots, h_n]\subset A_0^W\subseteq (A_0\otimes\End\,U)*\Wh$.
\end{prop}
\begin{cor}\label{cormain}
The elements $h_r(\lambda, y)$ admit a decomposition
\begin{equation}\label{spinsep0}
h_r(\lambda, y)=H_r+\hbar \,a_r   
\end{equation} 
for a suitable {(scalar)} quantum eCM Hamiltonian $H_r\in \Ah^W$ and some {matrix-valued} $a_r\in (\Ah\otimes\End\, U)*\Wh$ which is regular near $\lambda=0$.
\end{cor}
\begin{proof}
We abbreviate $h_r:=h_r(\lambda,y)$. By Proposition \ref{main}, $\eta_0(h_r)\in A_0^W$ is a classical eCM Hamiltonian, so it can quantised to a suitable quantum scalar eCM Hamiltonian $H_r\in\Ah^W$. Then $h_r$ and $H_r$ share the same classical limit, so we must have $h_r-H_r\in \hbar\,(\Ah\otimes\End\, U)*\Wh$.
\end{proof}

\subsection{}
To construct $R$-matrix eCM Hamiltonians, we use a version of the {`restriction'} map due to Heckman \cite{He}:
\begin{equation}\label{reshat}
{\widehat\Res}\colon \D_U(V)*\Wh\,\to\, \D_U(V)\,,\qquad \sum_{w\in W}a_{w}\,\hw 
\mapsto \sum_{w\in W}a_{w}\,. 
\end{equation}
The group $W$ acts on $\D_U(V)$ and $\D_U(V)*\Wh$ by \eqref{act}, \eqref{act1}. It
is easy to check that the map ${\widehat\Res}$ is $W$-equivariant, so it can be restricted onto $W$-invariants. The following result is standard (see, e.g., \cite[Lemma 4.1]{C24}).   
\begin{lemma}\label{alg}
The restriction of the map \eqref{reshat} onto $W$-invariants induces and algebra homomorphism ${\widehat\Res}\colon (\D_U(V)*\Wh)^{W}\to \D(V)^{W}$.
\end{lemma}
We can also view \eqref{reshat} as a map between $(\Ah\otimes\End\,U)*\Wh$ and $\Ah\otimes\End\, U$, which induces an algebra homomorphism
\begin{equation}\label{reshat1}
{\widehat\Res}:\ \left((\Ah\otimes\End\,U)*\Wh\right)^W\,\to\, (\Ah\otimes\End\, U)^W\,. 
\end{equation}
Combining this algebra map with Proposition \ref{main} gives the following result.

\begin{prop}\label{mainsp} Consider the elements $h_r(0,y)$ constructed in Proposition \ref{main} and define $R$-matrix eCM Hamiltonians by
$\rH_r={\widehat\Res}(h_r(0,y))$. Then $\rH_1, \dots, \rH_n$ are pairwise commuting $W$-invariant elements of $\Ah\otimes\End\, U$, with $\rH_r=\sigma_r(\hat p)+\ldots$, up to lower-order terms in $\hat p_i$. The classical limit of $\rH_r$ is a usual (scalar) classical eCM Hamiltonian, i.e., an element of $\c[h_1,\dots, h_n]$.       
\end{prop}  

\begin{cor}
\label{spinsep}
Each of the $R$-matrix eCM Hamiltonians can be written as 
\begin{equation}\label{spinsep1}
\rH_r=H_r+\hbar \,\widehat A_r    
\end{equation} 
for a suitable quantum (scalar) eCM Hamiltonian $H_r$ and some $\widehat A_r\in \Ah\otimes\End\, U$.
\end{cor}
\begin{proof}
This follows by applying $\widehat\Res$ to \eqref{spinsep0} and passing to the limit $\lambda\to 0$. 
\end{proof}
This tells us that the spin Hamiltonian $\rH_r$ can be approximated by a quantum scalar eCM Hamiltonian, so all the ``spin'' terms are of order $O(\hbar)$ (cf. \cite[Proposition~1.2]{GZ}). This will be important in order to construct spin chain Hamiltonians below using the so-called freezing procedure.

\subsection{}
As an illustration, let us look at the case of the quadratic Hamiltonian.
Making the substitution \eqref{subs} into the Hamiltonian \eqref{cmc} gives
\begin{equation*}
h(\lambda, y)=y_{1}^2+\dots +y_{n}^2-2\,{\g}^2\sum_{i<j}^n\wp(\lambda_i-\lambda_j)\,.
\end{equation*}
By a direct calculation and using the unitarity~\ref{it:unitarity}, skew-symmetry~\ref{it:skew} and the associative Yang--Baxter equation~\ref{it:aybe} for $R$, we obtain
\begin{equation}\label{y2}
h(\lambda,y)= \sum_{i=1}^n \hat p_i^2  
-2\,{\g}^2\sum_{i<j}^n\wp(x_i-x_j)- 2\,\g\,\hbar \sum_{i<j}^n \partial R_{ij} \, \widehat{s}_{ij}.
\end{equation}
Here $\partial R(z, \mu)=\frac{\partial}{\partial z}R(z, \mu)$. From \eqref{expR} we see that $\partial R(z, \mu)$ is regular at $\mu=0$, with $\partial R(z,0)=r'(z)$. Hence, in the limit $\lambda\to 0$ we get
\begin{equation}\label{y2r}
h(0, y) = \sum_{i=1}^n \hat{p}_i^2-2\,{\g}^2\sum_{i<j}^n\wp(x_i-x_j)-2\,\g\,\hbar\sum_{i< j}^n r'_{ij} \, \widehat{s}_{ij}\,.
\end{equation}
This is a $W$-invariant element of $(\Ah\otimes\End\,U)*\Wh$. The classical limit of \eqref{y2} is
\begin{equation}\label{y2c}
h(\lambda,y^{\cc})= \sum_{i=1}^n p_i^2  
-2\,{\g}^2\sum_{i<j}^n\wp(x_i-x_j).
\end{equation}
We see that it
is indeed independent of $\lambda$ and coincides with the classical scalar eCM Hamiltonian $h(x,p)$ from \eqref{cmc}. By applying the map \eqref{reshat} to \eqref{y2r}, we obtain the following quadratic $R$-matrix eCM Hamiltonian (cf. \cite[(3.15)]{GZ}:
\begin{equation}\label{reCMS}
    \rH={\widehat\Res}(h(0,y))=\sum_{i=1}^n \hat{p}_i^2-2\,{\g}^2\sum_{i<j}^n \wp(x_i-x_j) -2\,\g\,\hbar\sum_{i< j}^n r'_{ij}\,.
\end{equation}
This can be rearranged as 
\begin{equation}\label{spinsep2}
    \rH=\sum_{i=1}^n \hat{p}_i^2-2\,{\g}\,(\g-\hbar)\sum_{i<j}^n\wp(x_i-x_j)-2\,\g\,\hbar\sum_{i< j}^n \bigl(r'_{ij}+\wp(x_i-x_j)\bigr)\,,
\end{equation}
in an agreement with Corollary \ref{spinsep}.

\subsection{}
Let us now make the substitution \eqref{subs} into the cubic classical eCM Hamiltonian $h_3(x,p)$. Thus, we need to evaluate
\begin{equation}
h_3(\lambda, y)=\sum_{i<j<k} \! y_i \, y_j \, y_k + \frac{{\g}^2}{2} \! \sum_{i\ne j\neq k\neq i} \! \wp(\lambda_i-\lambda_j) \, y_k\,.
\end{equation}
Skipping intermediate steps, we obtain after taking $\lambda\to 0$ limit, 
\begin{align*}
h_3(0, y)=&\sum_{i<j<k} \! \pp_i \, \pp_j \, \pp_k+\frac{1}{2}\,{\g}^2 \!\!\! \sum_{i\neq j\neq k\neq i} \!\!\! \wp(x_i-x_j)\,\pp_k
\\&+\hbar \, {\g} \! \sum_{i<j<k}(r'_{ij} \, \pp_k \, \widehat{s}_{ij}+ r'_{jk} \, \pp_i \, \widehat{s}_{jk}+ r'_{ki} \, \pp_j \, \widehat{s}_{ki})
\\&+\hbar \, {\g}^2 \! \sum_{i<j<k} \! \left(r'_{ij}\,r_{ik}-r_{jk}\,r'_{ij}-m'_{ij}\right)\widehat{(ijk)}
\\&+\hbar \, {\g}^2 \! \sum_{i<j<k} \! \left(r'_{ij}\,r_{jk}-r_{ik}\,r'_{ij}+m'_{ij}\right){\widehat{(kji)}}\,,
\end{align*}
where $r'(z)=\partial_z r(z)$, $m'(z)=\partial_z m(z)$.
From this, we obtain $\rH_3=\widehat\Res( h_3(0,y))$:
\begin{multline*}
\rH_3=\sum_{i<j<k} \! \left( \pp_i \, \pp_j \, \pp_k+{\g}^2 \, \wp(x_i-x_j)\,\pp_k+{\g}^2 \, \wp(x_j-x_k)\,\pp_i+{\g}^2 \, \wp(x_k-x_i)\,\pp_j\right)
\\+\hbar \, {\g} \! \sum_{i<j<k} \! \left(r'_{ij} \, \pp_k \, + r'_{jk} \, \pp_i \,+ r'_{ki} \, \pp_j \right)
+\hbar \, {\g}^2 \! \sum_{i<j<k}[r'_{ij},\,r_{ik}+r_{jk}]\,.
\end{multline*}
In the classical limit this expression reduces to a scalar eCM Hamiltonian, in agreement with Proposition \ref{main}. We can also decompose it as in \eqref{spinsep1} into $\rH_3=H_3+\hbar \, \widehat A_3$ where
\begin{equation}
    H_3=\sum_{i<j<k} \! \bigl( \pp_i \, \pp_j \, \pp_k+{\g} \, (g-\hbar)\, (\wp(x_i-x_j)\,\pp_k+\wp(x_j-x_k)\,\pp_i+\wp(x_k-x_i)\,\pp_j)\bigr)
\end{equation}
is a quantum eCM Hamiltonian and
\begin{multline}\label{spinsep3}
    \widehat A_3=\, {\g} \sum_{i<j<k} \! \left(r'_{ij} \, \pp_k \, + r'_{jk} \, \pp_i \,+ r'_{ki} \, \pp_j \right)
 \, +{\g}^2 \! \sum_{i<j<k}[r'_{ij},\,r_{ik}+r_{jk}]\\
+{\g} \! \sum_{i<j<k}\bigl(\wp(x_i-x_j)\,\pp_k+\wp(x_j-x_k)\,\pp_i+\wp(x_k-x_i)\,\pp_j\bigr)\,.
\end{multline}

\remark{The quantum Hamiltonian \eqref{reCMS} was proposed in \cite{GZ}. Under our conventions, in the classical limit it reduces to the usual (scalar) eCM Hamiltonian. To see its connection to a classical system of interacting spinning tops, a more subtle version of the classical limit procedure is needed, see the discussion in  \cite[Section~4]{GZ}.}

\remark{Different versions of the quantum and classical spin Calogero--Moser systems were introduced and studied in recent works \cite{ReSt20, ReSt22, Re23}. Interestingly, their construction involves certain first-order differential operators, called asymptotic boundary KZB operators, which can be expressed via Felder's trigonometric dynamical classical $r$-matrices. Note that, unlike for the $R$-matrix Dunkl operators $y_i$, the commutativity of the asymptotic boundary KZB operators is highly non-trivial.}

\section{Proof of Proposition \ref{main}} \label{prf}

\noindent
\textit{Proof of part (1)}. We can follow the proof of \cite[Sec.~5.3]{EFMV}. It only requires us to check that for any root $\alpha=e_i-e_j$, the combinations $(\lambda_i-\lambda_j)\,y_{\alpha}$ and $(y_\alpha)^2-\langle\alpha,\alpha\rangle\,\wp(\lambda_i-\lambda_j)$ are regular near the hyperplane $\lambda_i=\lambda_j$. For $(\lambda_i-\lambda_j)\,y_{\alpha}=(\lambda_i-\lambda_j)(y_i-y_j)$ this is obvious. Next,
\begin{equation}\label{term}
  (y_\alpha)^2-\langle\alpha,\alpha\rangle\,\wp(\lambda_i-\lambda_j)=(y_i-y_j)^2-4\,g^2\,\wp(\lambda_i-\lambda_j)\,.  
\end{equation}
We can use the expression \eqref{y2} for $h(\lambda, y)$. Since $y_k$ for $k\ne i,j$ are regular at $\lambda_i=\lambda_j$, it shows that 
\[
y_i^2+y_j^2-2\,g^2\,\wp(\lambda_i-\lambda_j)=\frac12(y_i-y_j)^2+\frac12(y_i+y_j)^2-2\,g^2\,\wp(\lambda_i-\lambda_j)
\]
is regular at $\lambda_i=\lambda_j$. Because $y_i+y_j$ is regular at $\lambda_i=\lambda_j$, this gives us that \eqref{term} is also regular.

For the $W$-invariance, we note that 
\begin{equation}\label{inv}
\hw. h_r(\lambda, y(\lambda))=h_r(w.\lambda, y(w.\lambda))\,,    
\end{equation}
by $W$-invariance of $h_r$ and the equivariance of the Dunkl operators. In the limit $\lambda\to 0$ this gives $\hw.h_r(0,y)=h_r(0,y)$, as needed.
\smallskip

\textit{Proof of part (2)}. As $h_r(\lambda, y)$ is regular near $\lambda=0$, its classical limit $h_r(\lambda, y^\cc)$ is also regular at $\lambda=0$. Let us show that it is \emph{globally} regular. Possible singularities are along the hyperplanes $\lambda_i-\lambda_j=m+n\,\tau$ with $m,n\in\Z$. To rule them out, let us see how the classical Dunkl operator
\begin{equation*}
y_\xi^{\cc}(\lambda)=p_\xi-g\sum_{i<j}^n (\xi_i-\xi_j) \, R_{ij}(x_i-x_j,\lambda_i-\lambda_j) \, \widehat s_{ij}    
\end{equation*}
behaves under translations $\lambda\mapsto \lambda+u+\tau \,v$ with $u,v\in \Z^n$. 
Shifting $\lambda\mapsto \lambda+e_i$ or $\lambda\mapsto \lambda+e_i\,\tau$ only affects the terms $R_{ij}(x_i-x_j,\lambda_i-\lambda_j) \,\widehat s_{ij}$, $j\ne i$. Using \eqref{trans}, we find that
\[
R_{ij}(x_i-x_j,\lambda_i-\lambda_j+1) \, \widehat s_{ij}= Q_i^{-1}R_{ij}(x_i-x_j,\lambda_i-\lambda_j) \, \widehat s_{ij} \, Q_i\,,
\]
where $Q_i$ denotes the operator $Q$ acting in the $i$th tensor factor of $U$. Note that $Q_i$ commutes with any term $R_{kl} \, \widehat s_{kl}$ with $k,l\ne i$. Therefore,
\begin{equation*}
 y_\xi^{\cc}(\lambda+e_i)= Q_i^{-1} \, y_\xi^{\cc}(\lambda) \, Q_i\,.   
\end{equation*}
Similarly, 
\begin{equation*}
R_{ij}(x_i-x_j,\lambda_i-\lambda_j+\tau) \, \widehat s_{ij}= \exp\!\left(-\frac{2\pi\mathrm{i}\,(x_i-x_j)}{d}\right)\Lambda_i^{-1}R_{ij}(x_i-x_j,\lambda_i-\lambda_j)\,\widehat s_{ij}\,\Lambda_i\,,
\end{equation*}
which implies
\begin{equation*}
    y_\xi^{\cc}(\lambda+e_i\,\tau)=
    \exp\!\left(-\frac{2\pi\mathrm{i}\,x_i}{d}\right)\Lambda_i^{-1}
    y_\xi^{\cc}(\lambda) \, \Lambda_i
    \exp\!\left(\frac{2\pi\mathrm{i}\,x_i}{d}\right)\,.
\end{equation*}
The Hamiltonian $h_r$ has elliptic coefficients, hence $h_r(\lambda,y)$ has the same translations properties in $\lambda$ as $y_\xi^\cc$, namely, 
\begin{align}\la{trp1}
h_r(\lambda, y^{\cc}) & \: \xrightarrow{\lambda\mapsto \lambda+e_i} \ Q_i^{-1}
h_r(\lambda, y^{\cc}) \, Q_i^{-1}\,,
\\
\la{trp2}
h_r(\lambda, y^{\cc})\  &\xrightarrow{\lambda\mapsto \lambda+e_i\tau} \ 
\exp\!\left(-\frac{2\pi\mathrm{i}\,x_i}{d}\right)\Lambda_i^{-1}h_r(\lambda, y^{\cc}) \, \Lambda_i
    \exp\!\left(\frac{2\pi\mathrm{i}\,x_i}{d}\right)
\,.
\end{align}
Since we know that $h_r(\lambda, y^{\cc})$ is regular along the hyperplanes $\lambda_i-\lambda_j=0$, it must be regular along the shifted hyperplanes $\lambda_i-\lambda_j=m+n\,\tau$ with $m,n\in\Z$. As a result, it is regular everywhere, as needed.
Iterating \eqref{trp1}--\eqref{trp2} $d$ times we get $Q_i^{\mspace{2mu}d} =\Lambda_i^d=1$ and
\begin{equation}\label{trp3}
h_r(\lambda, y^{\cc}) \: \xrightarrow{\lambda\mapsto \lambda+de_i}
h_r(\lambda, y^{\cc}) \,,\qquad    h_r(\lambda, y^{\cc}) \  \xrightarrow{\lambda\mapsto \lambda+d\,e_i\tau} \ 
e^{-2\pi\mathrm{i}\,x_i} \, h_r(\lambda, y^{\cc})
    e^{2\pi\mathrm{i}\,x_i}\,.
\end{equation}
Now let us expand $h_r(\lambda, y^{\cc})$ as 
\begin{equation*}
h_r(\lambda, y^{\cc})= \sum_{w\in W} a_w \, \hw\,,\qquad a_w\in A_0\otimes\End\, U.  
\end{equation*}
Each coefficient $a_w$ is an $\End\, U$-valued function of $x, p, \lambda$, and it is globally holomorphic in $\lambda$. From \eqref{trp3} we see that, as a function of $\lambda$, $a_w$ is quasi-periodic with respect to the lattice $d \, (\Z^n\oplus \tau \Z^n)$. However, a holomorphic quasi-periodic function must be constant, which proves that each $a_w$ is constant in $\lambda$. And since \eqref{trp3} holds for all $i$, we must have $a_w=0$ for $w\ne\id$. As a result, $h_r(\lambda, y^{\cc})=a_{\id}$ is an element of $A_0\otimes\End\, U$. 

It remains to show that $a_\id$ is a multiple of $\mathrm{Id}_U$, i.e., lies in $A_0$. To this end, note that an expansion of $(y_i^{\cc})^k$ for any $k\in\N$ consists of summands of the form
\begin{equation}\label{summ}
    p_{i_1} \!\dotsb p_{i_l} \, R_{a_1b_1}\widehat{s}_{a_1b_1} \!\dotsb R_{a_mb_m}\widehat{s}_{a_mb_m}\,.
\end{equation}
We are interested in the terms contributing to $a_\id$, i.e., those with $\widehat{s}_{a_1b_1}\!\dotsb \widehat{s}_{a_mb_m}=\id$. The group $W$ is a Coxeter group, hence one can transform $\widehat{s}_{a_1b_1}\!\dotsb \widehat{s}_{a_mb_m}$ into the identity by using the braid relations $s_{ij}\,s_{jk}\,s_{ij}=s_{jk}\,s_{ij}\,s_{jk}$ and $s_{ij}^2=\id$. Since $R$ satisfies the quantum Yang--Baxter equation~\eqref{eq:YBE}, the elements $R_{ij}\,\widehat{s}_{ij}$ satisfy the braid relations and also 
$(R_{ij}\,\widehat{s}_{ij})^2=(\wp(\lambda_i-\lambda_j)-\wp(x_i-x_j))\,\mathrm{Id}_U$, by the unitarity. From this, it follows that the coefficients of the form \eqref{summ} contributing to $a_\id$ must all be scalars. Moreover, this argument also shows that this scalar does not depend on the spin dimension $d$, hence it is the same as in the $d=1$ case. In that case, $y_i$ are the usual (scalar) elliptic Dunkl operators and $h_r(0,y^\cc)$ represents a global section of the (classical) elliptic spherical Cherednik algebra, hence,
it must be a polynomial in $h_1, \dots, h_n$. See \cite[Sec.~6.3, 6.4]{EFMV} for the details (in particular, Theorem 6.3 of \textit{loc.~cit.}). \qed 

\begin{remark}\label{lf}
 {In \cite{C24}, some additional Hamiltonians for the spin eCM system were constructed using the \emph{spherical Cherednik algebra}. Recall that for $W=S_n$ and $g\in\c$, we have a Poisson subalgebra $\BB_g\subset A_0^W$, called the classical rational spherical Cherednik algebra, see \cite[Sec.~2.5]{C24} and \cite[Chapter 7]{E} for further details and references. It is known that $\BB_g$ is a universal Poisson deformation of $\c[V\times V^*]^W$. 
It turns out that for $f\in \BB_g$, the elements $f(\lambda, y)$ are regular at $\lambda=0$, hence $f(0, y)$ are well-defined (this is proved in the same way as Proposition \ref{main}(1)). We can now define, for any $f\in \BB_g$, $\L_f:=\widehat \Res(f(0, y))$. These are $W$-invariant elements of $\D_U$ which by construction commute with any $\Hsf_r$ and with each other. However, calculations in a few small examples show that, unlike in \cite{C24}, these $\L_f$ are trivial. We do not know if this is true in small examples only or for some general reason. For concrete examples, see Remark \ref{lcheckf} below.  
}   
\end{remark}

\section{{\it R}-matrix Lax pairs}

\subsection{}We can use the operators $y_i$ to re-derive the $R$-matrix Lax pairs found by Levin, Olshanetsky and Zotov in \cite{LOZ}.
We will use the approach of \cite{C19}, first constructing a quantum Lax pair and then passing to the classical limit. 
We will work in a specific representation of $\D_U(V)*\Wh$. Namely, let us view $\c(V)\otimes U$ as a left $\D_U(V)=\D(V)\otimes \End\,U$-module with the usual action by $\End\,U$-valued differential operators, and consider the induced module
\begin{equation*}
M=\Ind_{\D_U(V)}^{\D_U(V)*\Wh}\,\c(V)\,.
\end{equation*}

We can write elements of $M$ as $f=\sum_{w\in W} \widehat w \, f_w$ with $f_w\in\c(V)\otimes U$, thus identifying $M$ and $\c \Wh\otimes (\c(V)\otimes U)$ as vector spaces. The algebra $\End_\c(M)$ is then identified with $\End_\c(\c \Wh)\otimes \End_\c(\c(V)\otimes U)$, i.e., with matrices of size $|W|$ whose entries are operators acting on $\c(V)\otimes U$. As a result, the (left) action of $\D_U(V)*\Wh$ on $M$ gives a faithful representation
\begin{equation}\la{rep0}
\D_U(V)*\Wh\to \Mat(|W|, \D_U(V))\,.
\end{equation}
For any $W$-invariant $a\in\D_U(V)^W$ we have:
$a\left(\sum_{w\in W} \widehat w \,f_w\right)=\sum_{w\in W} \widehat w \, (a \, f_w)$.
Therefore, in the above representation such $a$ act as $a \, \mathbb{1}$.

\medskip
Now pick a Dunkl operator $y_\xi$ and any particular classical eCM Hamiltonian $h_r(x,p)$. Obviously, $y_\xi$ commutes with $h_r(\lambda,y)$. 
By \eqref{spinsep0}, 
\begin{equation}\la{acal}
 h_r(\lambda,y)=H_r+\hbar\, a_r\,,
\end{equation}
where $H_r$ is a quantum (scalar) eCM Hamiltonian, and $a_r\in (\Ah\otimes\End\, U)*\Wh$. The element $a_r$ depends on $\lambda$ but is regular near $\lambda=0$.

 Let $\mathcal L$, $\mathcal H_r$, $\mathcal A_r$ be the matrices representing under \eqref{rep0} the action of $y_\xi$, $H_r$ and $a_r$, respectively. Since $H_r$ is $\widehat W$-invariant, the matrix $\mathcal H_r$ is diagonal, $\H_r=H_r \, \mathbb{1}$. Then the commutativity $[y_\xi, h_r(\lambda,y)]=0$ leads to
\begin{equation}\label{leqg}
[\mathcal L, \mathcal H_r+ \hbar \mathcal A_r]=0\,,\quad \text{or}\quad 
\hbar^{-1}[H_r\,\mathbb{1},\, \mathcal L]=[\mathcal L, \mathcal A_r]\,.
\end{equation}
This is referred to as a quantum Lax equation. Passing to the classical limit, we
write $L:=\eta_0(\mathcal L)$, $A_r:=\eta_0(\mathcal A_r)$ and obtain
\begin{equation}
    \dot L=[L, A_r]\,,
\end{equation}
where $\dot{L}$ denotes the time-derivative with respect to the classical Hamiltonian flow defined by $\eta_0(H_r)$. Therefore, we have obtained a compatible family of quantum and classical Lax pairs $\mathcal L, \mathcal A_r$ and $L, A_r$, $r=1,\dots, n$, given by matrices of size $|W|=n!$\,.

\subsection{}
To get Lax pairs of size $n\times n$, we choose $y_\xi=y_1$, i.e., $\xi=(1,0,\dots, 0)$. The stabiliser $W'\subseteq W$ of $\xi$ is the symmetric group on the letters $\{2,\dots, n\}$, $W'\cong S_{n-1}$. Introduce
\begin{equation*}
e'=\frac{1}{(n-1)!}\sum_{w\in W'}\hw\,.
\end{equation*}
Choose $\lambda^0=(\mu,0,\dots,0)$ sufficiently close to $\lambda=0$. It has the same stabiliser as $\xi$. We can specialize $\lambda$ to $\lambda^0$ in \eqref{acal}, because all the terms are regular near $\lambda=0$. 
Also, from the definition of $y_1$ it is obvious that it is regular at $\lambda=\lambda^0$ and  specialises to
\begin{equation}\label{y1}
 y_1=\hat p_1 - {\g} \! \sum_{j (\ne 1)}^n \! R_{1j}(x_1-x_j, \mu) \, \widehat{s}_{1j}\,.
\end{equation}
The equivariance of the Dunkl operators implies that such $y_1$ is partially symmetric, namely, $\hw y_1 =y_1 \hw$ for any $w\in W'$. From \eqref{inv} with $\lambda=\lambda^0$, it is clear that the left-hand side of \eqref{acal} is also $\widehat W'$-invariant, and since $H_r$ is $\widehat W$-invariant, we get that $a_r$ is $\widehat W'$-invariant, that is, $\hw a_r=a_r\hw$ for all $w\in W'$. As a result, both $y_1$ and $a_r$ can be restricted on the subspace $M'=e'M$.
Elements of $M'$ can be written as
\begin{equation}
f=\sum_{i=1}^n e' \, \widehat s_{1i}f_i \,,\qquad f_i\in\c(V)\otimes U\,,    
\end{equation}
where $s_{11} \coloneqq \id$.
We can think of such $f$ as an $n$-tuple of elements in $\c(V)\otimes U$. 
This identifies $M'$ with $(\c(V)\otimes U)^{\oplus n}$, and therefore all the terms in \eqref{leqg} restrict to $n\times n$ (operator-valued) matrices acting on $M'$. Therefore, we obtain a compatible family of quantum/classical Lax pairs $\mathcal L, \mathcal A$ and $L, A$ of size $n\times n$, with the same Lax matrices $\mathcal L$ and $L$ but with  
different $\mathcal A=\mathcal A_r$, $A=A_r$ for each of the eCM Hamiltonians.

\subsection{} Let us follow the procedure above and construct a Lax pair for the Hamiltonian $\frac12h(x,p)=\frac12h_1^2-h_2$. Using (\ref{y2}) we have $\frac12h(\lambda,y)=H+\hbar\, a$, where 
\begin{equation}
H=\frac12\sum_{i=1}^n \hat p_i^2  
-{\g}\,(g-\hbar)\sum_{i<j}^n\wp(x_i-x_j)\,,\quad 
a=  
-{\g} \sum_{i<j}^n\wp(x_i-x_j)- {\g} \sum_{i<j}^n \partial R_{ij} \, \widehat{s}_{ij}.
\end{equation}
Under the specialisation $\lambda=\lambda^0$, the $a$-term becomes
\begin{equation}\label{a}
a=-{\g} \sum_{i<j}^n\wp(x_i-x_j)-{\g} \! \sum_{j (\ne 1)}^n \! \partial R_{1j}(x_1-x_j,\mu) \, \widehat{s}_{1j}-\,{\g}\!\sum_{1<i<j}^n \!r'_{ij} \, \widehat{s}_{ij}\,.
\end{equation}
To find the quantum Lax pair $\mathcal L, \mathcal A$, we need to calculate the action of $y_1$ \eqref{y1} and $a$ \eqref{a} on $M'=e'M$. Note that because $\widehat{s}_{ij}=\id$ on $M'$ for $i,j>1$, $a$ can be replaced by
\begin{equation}\label{a1}
-\, \,{\g} \sum_{i<j}^n\wp(x_i-x_j)-\,\,{\g} \! \sum_{j (\ne 1)}^n \! \partial R_{1j}(x_1-x_j,\mu) \, \widehat{s}_{1j}-\,{\g}\!\sum_{1<i<j}^n \!r'_{ij}\,.
\end{equation}
Following the recipe in \cite[Lemma~2.3]{C19}, we obtain the $n\times n$ operator-valued matrices $\mathcal L, \mathcal A$ in the form
\begin{equation*}
\mathcal L_{ij}=
\begin{cases}
-{\g} \, R_{ij}(x_i-x_j, \mu) \ &\text{for}\ i\ne j\\
\hat p_i\ &\text{for}\ i=j
\end{cases}\,\,,\qquad 
\end{equation*}
\begin{equation*}
\mathcal A_{ij}=
\begin{cases}
-\, {\g}\,\partial R_{ij}(x_i-x_j, \mu) \ &\text{for}\ i\ne j\\
-\,{\g}\left(\sum_{k<l}^n\wp(x_k-x_l)+\sum_{\substack{k<l \\ k,l\neq i}} r'_{kl}\right) &\text{for}\ i=j
\end{cases}\,\,,
\end{equation*}
which coincides with the quantum Lax pair found in \cite[Proposition 3.1]{GZ}.
The classical Lax pair is obtained in the classical limit, resulting in
\begin{equation*}
L_{ij}=
\begin{cases}
-{\g} \, R_{ij}(x_i-x_j, \mu) \ &\text{for}\ i\ne j\\
p_i\ &\text{for}\ i=j
\end{cases}
\end{equation*}
and $A$ identical to $\mathcal A$. Both matrices depend on the spectral parameter, $\mu$. The scalar matrix $-g\sum_{k<l}^n\wp(x_k-x_l) \, \mathrm{Id}_n$ can be removed from $A$ since it commutes with $L$. Hence, we may take
\begin{equation*}
A_{ij}=
\begin{cases}
-\, {\g}\,\partial R_{ij}(x_i-x_j, \mu) \ &\text{for}\ i\ne j\\
-\,{\g}\sum_{\substack{k<l \\ k,l\neq i}} r'_{kl} &\text{for}\ i=j
\end{cases}\,.
\end{equation*}
Note that $A_{ii}=-g\sum_{k<l}^n r'_{kl}+g\sum_{k(\ne i)}^nr'_{ik}$, however the first sum cannot be removed because it represents the diagonal matrix $-g \, (\sum_{k<l}^n r'_{kl}) \, \mathrm{Id}_n$ which does not commute with $L$.
This classical Lax pair $L, A$ matches the $R$-matrix Lax pair found in \cite{LOZ}.

\section{{\it R}-matrix quantum spin chain from freezing}\label{sc0}

To pass from the $R$-matrix Hamiltonians to a quantum spin chain, we apply the procedure known as ``freezing'', which involves placing classical particles at an equilibrium. This heuristic method goes back to Polychronakos \cite{P}. Recently, it was put on a firm ground to demonstrate integrability of the Inozemtsev spin chain in \cite{C24}; it was also given a broader treatment and justification within the framework of hybrid integrable systems in \cite{LRS}.
We will follow the approach of \cite{C24}.

\medskip

Recall that in \cite{SZ}, Sechin and Zotov proposed the following spin chain Hamiltonian $\H_2\in\End\, U$:
\begin{equation}\label{h2}
\H_2=\sum_{i<j}\bar r_{ij}\,,\qquad \bar r_{ij}:=r_{ij}\left(\frac{i-j}{n}\right)\,.  \end{equation}
{An overview of these and related spin chains can be found in \cite{KL1}.} They further conjectured that $\H_2$ commutes with
\begin{equation}\label{h3}
\H_3=\sum_{i<j<k} [\bar r'_{ij},\bar r_{ik}+\bar r_{jk}]\,,\qquad \bar r'_{ij} \coloneqq r'_{ij}\left(\frac{i-j}{n}\right)\,.
\end{equation}

The commutativity $[\H_2, \H_3]=0$ was checked in \cite{SZ} numerically for $d=2$, and in \cite{MZ} it was deduced from the commutativity of relativistic analogues of $\H_2, \H_3$. 

We can now obtain this as part of the following construction of commuting spin Hamiltonians $\H_2, \dots, \H_n$. Recall that by \eqref{spinsep1}, $\rH_r=H_r+\hbar \, \widehat A_r$ . Here each $\widehat A_r$ is an element of $A_\hbar\otimes \End\, U$, which can be thought of as a matrix (acting on $U$) whose entries belong to the algebra of differential operators $A_\hbar$ \eqref{ah}. Its classical limit, $\eta_0(\widehat A_r)$, is an element of $A_0\otimes \End\, U$ where 
$A_0$ is the algebra of functions of $x_i, p_i$ \eqref{a0}. Hence, it can be evaluated at a particular point $(x^\star, p^\star)$. For our purposes, we need to evaluate at an equilibrium $(x^\star, p^\star)$ of $h(x,p)$ from \eqref{cmc} 
given by 
\begin{equation}
  x^*=(x^*_i)\,,\quad p^*=(p^*_i)\,,\qquad x_i^*=\frac{i}{n}\,,\quad p_i^*=0\,,\quad i=1,\dots, n.
\end{equation}
\begin{prop}
Define spin Hamiltonians $\H_r\in\End\, U$ by 
\begin{equation}
\mathcal H_r \coloneqq \eta_0(\widehat A_r)\big|_{(x,p)=(x^\star, p^\star)}\,,\quad r=2,\dots, n\,.
\end{equation}
Then we have $[\H_r, \H_s]=0$ for all $r,s$. 
\end{prop}
The proof repeats that of \cite[Theorem 5.5]{C24} verbatim, cf.\ \cite[Sec.~9]{LRS}. 
The reason to exclude $r=1$ is that in that case $\widehat A_1=0$ and so $\H_1=0$.
\qed

\medskip

For $r=2,3$ we recover from \eqref{spinsep2}, \eqref{spinsep3} the above $\H_2, \H_3$, up to an additive constant term (a constant multiple of $\mathrm{Id}_U$).

\section{Spectral deformation}

Here we allow spectral variables into Hamiltonians and generalise the previous constructions analogously to \cite[Sec. 7]{C24}.

\subsection{}First, let us construct a dynamical spin model incorporating dependence on $\lambda$.
It will be convenient to work with an additional copy $\Vc\cong V$ of the space $V$ with $\lambda\in \Vc$, and denote by $\Sc$ a copy of the symmetric group acting on $\Vc$.  
Denote by $\D_U^\vee(V)$ the algebra of matrix differential operators on $V$ whose coefficients also depend on $\lambda$; it is generated by (operators of multiplication by) functions $f\in \c(V\times \Vc)$ and the derivations $\partial_i=\partial/\partial x_i$. We have two commutative $W$-actions on $\D_U^\vee(V)$: the action \eqref{act} and the $\Sc$-action on $\lambda$, with $w^\vee.f(x, \lambda):=f(x, w^{-1}.\lambda)$ for $f\in \c(V\times \Vc)$. Hence, we can form the crossed product 
\begin{equation*}
\A^\vee \coloneqq \D_U^\vee(V)*(\Wh\times\Sc)\,.
\end{equation*}  	
Elements of $\A^\vee$ can be uniquely written as
\begin{equation*}
a=\! \sum_{w_1, w_2\in W} \!\!\! a_{w_1w_2} \, \hw_1\otimes w_2^{\vee}\quad\text{with}\ a_{w_1w_2}\in\D_U^\vee(V)\,.
\end{equation*}
Define a linear map 
\begin{align}\nonumber
{\Res}^\vee&:\ \A^\vee\,\to\, \D_U^\vee(V)*\Sc\,,\\ 
\label{reshatl}
&\sum_{w_1,w_2\in S_n}a_{w_1w_2} \, \hw_1\otimes w_2^\vee\ 
\mapsto \sum_{w_1,w_2\in S_n}a_{w_1w_2} \, (w_2 \, w_1^{-1})^\vee\,. 
\end{align}
Equivalently, ${\Res}^\vee(a)$ is the unique element $L_a\in\D_U^\vee(V)*\Sc$ such that
\begin{equation}
a \, e=L_a \, e\,,\qquad e \coloneqq \frac{1}{n!}\sum_{w\in W} \hw\otimes w^\vee\,.
\end{equation}
The group $W$ acts on $\A^\vee$ and $\D_U^\vee(V)*\Sc$ by conjugation, 
\begin{equation}\label{acc}
a\mapsto (\hw\otimes w^\vee) \, a \, (\hw\otimes w^\vee)^{-1}\qquad\forall\ w\in W\,.
\end{equation}
It is easy to check that the map ${\Res}^\vee$ is $W$-equivariant, hence, it can be restricted on the subspaces of $W$-invariants, analogously to Lemma \ref{alg}.   
\begin{lemma}
The restriction ${\Res}^\vee \colon (\A^\vee)^{W}\to (\D_U^\vee(V)*\Sc)^{W}$ of the map \eqref{reshatl} is an algebra homomorphism.
\end{lemma}

Now, for any $W$-invariant function $f(x,p)$ in $A_0$ \eqref{a0}, the result of substitution \eqref{subs} is a $W$-invariant element $f(\lambda, y)$ of $\D_U^\vee(V)*\Wh\subset \A^\vee$.
Indeed, \begin{equation}\label{inv1}
    (\hw\otimes w^\vee) \, f(\lambda, y) \, (\hw\otimes w^\vee)^{-1}=f(\lambda, y)\,,
\end{equation}
by equivariance of $y_i$ and $W$-invariance of $f(x,p)$, cf. \eqref{inv}. These elements pairwise commute because $y_i$ commute. Hence, the following result.

\begin{prop}
Define $\L^\vee_f:=\Res^\vee f(\lambda, y)$ for $f\in A_0^W$. The elements $\L^\vee_f$ are pairwise commuting $W$-invariant elements of $\D_U^\vee(V)*\Sc$.    
\end{prop}

\subsection{}\label{hcheck}
Among $\L^\vee_f$, we have analogues of the $R$-matrix eCM Hamiltonians. First, from \eqref{y2}, we find the following generalisation of the Hamiltonian \eqref{reCMS}: 
\begin{equation}
    \Hsf^\vee=\sum_{i=1}^n \hat p_i^2  
-2\,{\g}^2\sum_{i<j}^n\wp(x_i-x_j)- 2\,\g\,\hbar \sum_{i<j}^n \partial R_{ij}\, {s}^\vee_{ij}.
\end{equation}
Here $\partial R(z, \mu)=\frac{\partial}{\partial z}R(z, \mu)$, $\partial R_{ij}=\partial R_{ij}(x_i-x_j, \lambda_i-\lambda_j)$ and $s^\vee_{ij}$ acts only on $\lambda\in V^\vee$. In the limit $\lambda\to 0$, $s^\vee_{ij}=\id$ and we get the $R$-matrix Hamiltonian \eqref{reCMS}.

Next, define the \emph{principal Hamiltonians}  
\begin{equation}
    \Hsf^\vee_r:=\Res^\vee(h_r(\lambda, y))\,,\qquad r=1,\dots, n.
\end{equation}
Here $h_r(x,p)$ are the scalar eCM Hamiltonians, see Sec.~\ref{subsection:cecms}.
By a direct calculation, the first three principal Hamiltonians are:
\begin{align*}
\Hsf^\vee_1=&\hat p_1+\dots +\hat p_n,\\
\Hsf^\vee_2=&\sum_{i<j}^n\left(\hat p_i\hat p_j+g^2\wp(x_i-x_j)+g\hbar \,\partial R_{ij}\,s^\vee_{ij}\right),\\
\Hsf^\vee_3=&\sum_{i<j<k}(\pp_i\pp_j\pp_k+{\g}^2\wp(x_i-x_j)\pp_k+{\g}^2\wp(x_j-x_k)\pp_i+{\g}^2\wp(x_k-x_i)\pp_j)\\
+&
\hbar g\sum_{i<j<k}(\partial R_{ij}\pp_k {s}^\vee_{ij}+ \partial R_{jk}\pp_i{s}^\vee_{jk}+ \partial R_{ki}\pp_j{s}^\vee_{ki})
\\+&\hbar{\g}^2\sum_{i<j<k}\left(\partial R_{ij}\, R_{ik}^{jk}-R_{jk}\, \partial R_{ij}^{ik}\right){(ijk)^\vee}
+\hbar g^2\sum_{i<j<k}\left(\partial R_{ij}\,R_{jk}^{ik}-R_{ik}\,\partial R_{ij}^{kj}\right){(kji)^\vee}\,.
\end{align*}
In the limit $\lambda\to 0$, $\Hsf^\vee_r=\rH_r$, so we view $\Hsf^\vee_r$ as  \emph{deformations} of $\rH_r$.

\subsection{}\label{lcheck}
In contrast with the case $f=h_r(x,p)$, for an arbitrary $f\in A_0^W$ the expression $f(\lambda, y)$ would not be regular at $\lambda=0$. However, as we already explained in Remark \ref{lf}, one can use $f$ from the rational spherical Cherednik algebra $\BB_g\subset A_0^W$, and we have the following result. 

\begin{prop}\label{sphhamc}
The elements $\L^\vee_f=\Res^\vee(f(\lambda, y))$ for $f\in \BB_g$, are $W$-invariant elements of $\D_U^\vee(V)*\Sc$, regular at $\lambda=0$.
\end{prop}
Proof repeats the proof of Proposition \ref{main}(1). \qed

Here are two examples. First, consider
$f(x,p)=\sum_{i=1}^n x_i \, p_i\in \BB_g$. This gives 
\begin{equation*}
f(\lambda, y)=\sum_{i=1}^n \lambda_i \, y_i=\sum_{i=1}^n \lambda_i \left(\pp_i-{\g}\sum_{j\ne i}R_{ij} \, \widehat{s}_{ij}\right)\,,
\end{equation*}
and, as a result,
\begin{equation}\label{ham1}
\L^\vee_f=\sum_{i=1}^n \lambda_i \pp_i- g\sum_{i<j}(\lambda_i-\lambda_j)R_{ij} \,{s}^\vee_{ij}\,.
\end{equation}
Next, we take $f\in \BB_g$ as follows:
\begin{equation}\label{f2}
f(x,p)=\sum_{i\ne j\ne k\ne i}x_i\left(p_j \, p_k+\frac{{\g}^2}{(x_{j}-x_k)^2}\right)\,.    
\end{equation}
Then
\begin{equation}
 f(\lambda, y)=\sum_{i\ne j\ne k\ne i}\lambda_i\left(y_j \, y_k+\frac{{\g}^2}{(\lambda_{j}-\lambda_k)^2}\right)\,.   
\end{equation}
From this,
\begin{align}\nonumber
\L_f^\vee&=\sum_{i\ne j\neq k\neq i}\lambda_i(\pp_j\pp_k +{\g}^2\,\wp(x_j-x_k))+
{\g} \sum_{i\ne j\neq k\neq i}(\lambda_i-\lambda_j)R_{ij}\pp_k{s}^\vee_{ij}+
\\\nonumber
&
+{\g} \hbar\sum_{i\ne j\neq k\neq i}\lambda_i\, \partial R_{kj}s_{kj}^\vee
\\\nonumber&-2{\g}^2\sum_{i<j<k}\left((\lambda_k-\lambda_i)R_{ik}R_{jk}^{ji}+(\lambda_k-\lambda_j)R_{ij}R_{ik}^{jk}\right){(ijk)^\vee}
\\\label{ham2}&+2{\g}^2\sum_{i<j<k}\left((\lambda_i-\lambda_k)R_{ij}R_{jk}^{ik}+(\lambda_k-\lambda_j)R_{jk}R_{ik}^{ij}\right){(kji)^\vee}\,.
\end{align}

\subsection{}\label{lcheckspin}
Now the spin chain Hamiltonians incorporating $\lambda$-variables are constructed similarly to \cite[Sec. 7]{C24}. 
We view the Hamiltonians $\Hsf_r^\vee$ and $\L_f^\vee$ as elements of $(\Ah\otimes \End\, U) * \Sc$ so we can take classical limits. We have a counterpart of \eqref{spinsep1}:
\begin{equation}\label{decl}
\Hsf^\vee_r=H_r+\hbar\,\Asf^\vee_q\,,\qquad\text{with}\ \Asf^\vee_q\in (\Ah\otimes\End\, U)*\Sc\,.
\end{equation}
We define \emph{principal} spin chain Hamiltonians by
\begin{equation}
\mathcal H_r^\vee \coloneqq \eta_0(\Asf^\vee_r)\big|_{(x,p)=(x^*, p^*)}\,,\quad r=2,\dots, n\,.
\end{equation}
For $r=2,3$ we use the expressions for $\Hsf^\vee_2, \Hsf^\vee_3$ to find that, up to a constant factor,
\begin{align*}
\mathcal H_2^\vee&=\, \sum_{i<j}^n\partial \bar R_{ij}\,\!{s}_{ij}^\vee\,,\qquad 
\partial \bar R_{ij}:=\partial R_{ij}\left(\frac{i-j}{n}, \lambda_i-\lambda_j\right)\,
\\
\mathcal H^\vee_3&=\sum_{i<j<k}(\partial \bar R_{ij}\, \bar R_{ik}^{jk}-\bar R_{jk}\, \partial \bar R_{ij}^{ik}){(ijk)^\vee}
+\sum_{i<j<k}
(\partial \bar R_{ij}\,\bar R_{jk}^{ik}-\bar R_{ik}\,\partial \bar R_{ij}^{kj}){(kji)^\vee}\,.
\end{align*}
Here $\bar R_{ij}^{kl}=R_{ij}\left(\frac{i-j}{n}, \lambda_k-\lambda_l\right)$ and
$\partial \bar R_{ij}^{kl}=\partial R_{ij}\left(\frac{i-j}{n}, \lambda_k-\lambda_l\right)$. The Hamiltonians $\H^\vee_2, \H^\vee_3$ are $\lambda$-deformations of \eqref{h2}-\eqref{h3}. 

We can also define \emph{additional} Hamiltonians by using $\L_f^\vee$ as follows:
\begin{equation}
\mathcal I^\vee_f:=\eta_0(\L^\vee_f)\big|_{(x,p)=(x^*, p^*)}\,,\quad f\in \BB_g\,.
\end{equation}
For example, taking $f=\sum_{i=1}^n x_i \, p_i$ gives 
\begin{equation}\label{ham1sp}
    \mathcal I^\vee_f=\sum_{i<j}(\lambda_i-\lambda_j)\bar R_{ij}s^\vee_{ij}\,.
\end{equation}
For $f$ in \eqref{f2}, we find from the formula for $\L_f^\vee$ that
\begin{align}\nonumber
\mathcal I_f^\vee&=\sum_{i\ne j\neq k\neq i}\lambda_i\wp_{jk}
\\\label{ham2sp}&-2\sum_{i<j<k}\left((\lambda_k-\lambda_i)\bar R_{ik}\bar R_{jk}^{ji}+(\lambda_k-\lambda_j)\bar R_{ij}\bar R_{ik}^{jk}\right){(ijk)^\vee}
\\\nonumber&+2\sum_{i<j<k}\left((\lambda_i-\lambda_k)\bar R_{ij}\bar R_{jk}^{ik}+(\lambda_k-\lambda_j)\bar R_{jk}\bar R_{ik}^{ij}\right){(kji)^\vee}\,.
\end{align}
Here $\wp_{jk}=\wp(\frac{j-k}{n})$, and $\bar R_{ij}$, $\bar R_{ij}^{kl}$ have the same meaning as in the formulas for $\H^\vee_{2}$, $\H^\vee_3$ above.

\begin{theorem} \label{intc}
The elements $\H_r^\vee$, $r=2,\dots, n$ and $\mathcal I_f^\vee$ with $f\in \BB_g$ form a commutative family of elements of $(\c(\Vc)\otimes \End\, U)*S_n^\vee$.
\end{theorem}
\begin{proof}
The proof repeats verbatim the proofs in \cite[Sec.~5.3]{C24}, see in particular Theorems 5.5, 5.7 in {\it loc. cit.}.    
\end{proof}

\subsection{} To interpret the elements $\H_r^\vee$, $\mathcal I_f^\vee$ as Hamiltonians of a spin chain, one needs to choose a representation of $(\c(\Vc)\otimes\End\, U)*S_n^\vee$ by appropriately enlarging the spin space $U$. A natural choice is
the space $\c(\Vc)\otimes U$ but the drawback is that it is infinite-dimensional.
Note that the action on $\c(\Vc)$ only involves permutations $w^\vee$ for $w\in W$ and multiplication by functions of $\lambda\in\Vc$. Therefore, it can be restricted onto a finite, $W$-invariant set of $\lambda$'s. 

For example, take a generic point $\epsilon=(\epsilon_1,\dots, \epsilon_n)\in\Vc$. Its $W$-orbit consists of $|W|=n!$ points 
$\sigma.\epsilon=(\epsilon_{\sigma(1)},\dots, \epsilon_{\sigma(n)})$, $\sigma\in W$. Let us replace $\c(\Vc)\otimes U$ with the direct sum of $|W|$ copies of $U$:
\begin{equation}
    U^\vee: =
    \oplus_{\sigma\in W} U^{(\sigma)}\,,\qquad\text{with}\quad  U^{(\sigma)}\cong U\quad \forall\ \sigma\,.
\end{equation}
We make this space into a $(\c(\Vc)\otimes\End\, U)*S_n^\vee$-module, with $w^\vee$ acting by permuting the summands, $w^\vee:\, U^{(\sigma)}\overset{\sim}\mapsto U^{(w\sigma)}$, and with $f(\lambda)$ acting on each summand $U^{\sigma}\cong U$  by an operator $f(\sigma.\epsilon)\in\End\, U$. In other words,
any $f\in\c(V^\vee)\otimes \End\, U$ acts on $U^\vee$ by 
$\oplus_{\sigma\in W}f(\sigma.\epsilon)\mathrm{Id}_{U^{(\sigma)}}$. For instance, a summand in $\H^\vee_2$,   
\begin{equation}
 \partial \bar R_{ij}\,\!{s}_{ij}^\vee =\partial R_{ij}\left(\frac{i-j}{n}, \lambda_i-\lambda_j\right){s}_{ij}^\vee\,,   
\end{equation}
swaps $U^{(\sigma)}$ and $U^{s_{ij}\sigma}$ for each $\sigma$ and then applies $\partial R_{ij}\left(\frac{i-j}{n}, \epsilon_{\sigma(i)}-\epsilon_{\sigma(j)}\right)$ to each direct summand $U^{(\sigma)}\subset U^\vee$.   

\medskip

Because $\partial R_{ij}\left(\frac{i-j}{n}, \mu\right)$
is regular at $\mu=0$, we can degenerate the above setup and use non-generic $W$-orbits. Namely, take an arbitrary $\epsilon=(\epsilon_1,\dots, \epsilon_n)\in\Vc$ such that all the Hamiltonians are regular at $\lambda=\epsilon$. Let $\Sigma=\{\sigma\}\subset \Vc$ be the $W$-orbit of $\epsilon$. Consider the space 
\begin{equation}
    U^\vee: =
    \oplus_{\sigma\in \Sigma} U^{(\sigma)}\,,\qquad\text{with}\quad  U^{(\sigma)}\cong U\quad \forall\ \sigma\in\Sigma\,.
\end{equation}
We can make it into a $(\c(\Vc)\otimes\End\, U)*S_n^\vee$-module, with $w^\vee$ acting by $w^\vee:\, U^{(\sigma)}\overset{\sim}\mapsto U^{(w.\sigma)}$, and with $f(\lambda)$ acting on each summand $U^{\sigma}\cong U$  by an operator $f(\sigma)\in\End\, U$. In other words,
any $f\in\c(V^\vee)\otimes \End\, U$ acts on $U^\vee$ by 
$\oplus_{\sigma\in \Sigma}f(\sigma)\mathrm{Id}_{U^{(\sigma)}}$. 

For example, taking $\epsilon=(\mu,0,\dots, 0)$ with $\mu\ne 0$, its orbit has length $n$ so we can denote $\sigma_i:=(0,\dots, \mu,\dots, 0)$, with $\mu$ in the $i$th place. Then 
\begin{equation}
    U^\vee: =
    \oplus_{i=1}^n U^{(i)}\,,\qquad\text{with}\quad  U^{(i)}:=U^{(\sigma_i)}\cong U\quad \forall\ i\,.
\end{equation}
A permutation $w^\vee$ acts on $U^\vee$ by simply permuting the summands, $w^\vee:\,  U^{(i)}\overset{\sim}\mapsto U^{(w(i))}$, and any $\End\, U$-valued function $f(\lambda)$ acts on $U^{(i)}$ as $f(\sigma_i)$. Choosing $\mu=0$ gives the trivial $W$-orbit $\Sigma=\{0\}$, in which case we recover the spin chain described in Sec.~\ref{sc0}.  

\begin{remark}
    In the same way, one can specialise $\lambda$ in the dynamical Hamiltonians $\Hsf _r^\vee$ and $\mathsf L_f^\vee$ from Sec.~\ref{hcheck}, \ref{lcheck}.
\end{remark}

\begin{remark}\label{lcheckf} We have explicit examples of the additional Hamiltonians $\L^\vee_f$ and $\mathcal I_f$ calculated in Sec.~\ref{lcheck}, \ref{lcheckspin}. Let us see how they behave in the limit $\lambda\to 0$.
Using \eqref{expR}, we find that \eqref{ham1} at $\lambda=0$ is a constant multiple of $\sum_{i<j}s^\vee_{ij}$ which acts as a constant because $s^\vee_{ij}=\id$ when $\lambda$ is set to zero. 
Similarly, evaluating \eqref{ham2} at $\lambda=0$ gives, after some computations,
and replacement $w^\vee=\id$, a constant multiple of $\Hsf_1=\sum_{k}\pp_k$. As a consequence, the operator \eqref{ham1sp} reduces to a constant and \eqref{ham2sp} vanishes under $\lambda\to 0$. We do not know if this is a more general property or just happens in small examples. Note that for the Inozemtsev spin chain treated in \cite{C24}, $R(z,\mu)=\phi(z,\mu)P$ so its residue at both $\mu=0$ and $z=0$ is the permutation operator, $P$. This changes the expressions for the Hamiltonians $\L^\vee_f$ at $\lambda\to 0$, making them nontrivial. 
\end{remark}

\section{Trigonometric case}\label{trigcase}
\subsection{}The above constructions remain valid for the (non-degenerate) trigonometric solutions to the AYBE, classified in \cite{S, Pol2}. We refer the reader to \cite{Pol2} for their construction and explicit expressions, see also
\cite{KZ} for a discussion of interesting special cases. Each of these solutions is determined by a so-called Belavin--Drinfeld structure on a finite set $S$ of size $N$ \footnote{To match the AYBE in \cite[(0.1)]{Pol2} with the equation \eqref{eq:AYBE}, one needs to swap the variables and use the skew-symmetry.}.
Let $r(u,v)\in \End\,(\c^N\otimes\c^N)$ be as given in \cite[Theorem~0.1]{Pol2}\footnote{It is related to $R^\hbar(z)$ in \cite{KZ} by $R^\hbar(z)=\frac{N}{2}\, r(2\hbar/N, 2z/N)$.}. 
If we define 
\begin{equation}
    R(z, \mu):=\,\mathrm{i}\,r(\mathrm{i}\mu, \mathrm{i}z)\in \End\,(\c^N\otimes\c^N)\,,
\end{equation}
then from \cite[(0.1),(0.2),(0.6), (0.7)]{Pol2} we see that $R$ has all the properties \ref{it:skew}--\ref{it:regular}, with the unitarity \ref{it:unitarity} in the form 
\begin{equation}
R(z,\mu) \, R_{21}(-z,\mu) =\left(\frac{1}{4 \sin^{2}\frac{\mu}{2}}-\frac{1}{4\sin^{2}\frac{z}{2}} \right)\, \mathrm{Id}\,.    
\end{equation}
From the explicit formula \cite[Theorem 0.1]{Pol2} for $r(u,v)$ we also see that, as a function of $\mu$, $R(z,\mu)$ has a period $2\pi N$, with simple poles at $\mu\in 2\pi\Z$. By a direct check based on the explicit formula, 
\begin{equation}\label{transtrig}
    R(z,\mu+2\pi)=(Q^{-1}\otimes \mathrm{Id})\,R(z,\mu)\,(\mathrm{Id}\otimes  Q)
\end{equation}
for some constant diagonal $N\times N$ matrix $Q$ which depends on the Belavin--Drinfeld structure that determines the solution $r(u,v)$.   
 
With these ingredients, the definition \eqref{RD} of the $R$-matrix Dunkl operators $y_i$ remains the same. To construct $R$-matrix Calogero--Moser--Sutherland (CMS) Hamiltonians, we substitute $y_i$ into the classical CMS Hamiltonians which are obtained from the eCM Hamiltonians by replacing 
\begin{equation}
\wp(x_i-x_j)\mapsto \frac{1}{4 \sin^{2}\frac{x_i-x_j}{2}}\,.     
\end{equation}
In particular, a trigonometric analogue of \eqref{cmc} is 
\begin{equation}\label{cmcs}
h(x,p)=\sum_{i=1}^n p_i^2-2\,\g^2\sum_{i<j}^n\frac{1}{4 \sin^{2}\frac{x_i-x_j}{2}}\,.
\end{equation}
Similarly to \eqref{y2} we calculate 
\begin{equation}\label{y2cms}
h(\lambda,y)= \sum_{i=1}^n \hat p_i^2  
-2\,{\g}^2\sum_{i<j}^n \frac{1}{4 \sin^{2}\frac{x_i-x_j}{2}}- 2\,\g\,\hbar \sum_{i<j}^n \partial R_{ij} \, \widehat{s}_{ij}.
\end{equation}
Proposition \ref{main} remains unchanged, except that we need to use trigonometric Hamiltonians instead of the elliptic ones. The proof needs a modification, as detailed next.

\subsection{} \emph{Proof of part (1)} is the same as in the elliptic case, up to replacing 
\begin{equation*}
\wp(\lambda_i-\lambda_j)\mapsto \frac{1}{4 \sin^{2}\frac{\lambda_i-\lambda_j}{2}}\,.   
\end{equation*}

\emph{Proof of part (2)} begins similarly. In fact, we can see that $h(\lambda,y)$ is globally regular already at the quantum level. Indeed, from    
\eqref{transtrig},
\begin{equation*}
 y_\xi(\lambda+2\pi e_i)= Q_i^{-1} \, y_\xi(\lambda) \, Q_i\,\,,\qquad h_r(\lambda, y)  \: \xrightarrow{\lambda\mapsto \lambda+2\pi e_i} \ Q_i^{-1}
h_r(\lambda, y) \, Q_i^{-1}\,,   
\end{equation*}
from which the global regularity of $h_r(\lambda, y)$ as well as $h_r(\lambda, y^{\cc})$ follows. Now, we need the following fact.
\begin{lemma}
    Let $S_{ij}:=(\cot\frac{x_i-x_j}{2}+\cot\frac{\lambda_i-\lambda_j}{2})^{-1}R_{ij}(x_i-x_j, \lambda_i-\lambda_j)\hs_{ij}$. Then the elements $S_{ij}$ satisfy the relations $S_{ij}^2=1$, $S_{ij}S_{jk}S_{ij}= S_{jk}S_{ij}S_{jk}$ for distinct $i,j,k$. In addition, $S_{ij} S_{kl}= S_{kl}S_{ij}$ if $\{i,j\}\cap\{k,l\}=\emptyset$.  
\end{lemma}
\begin{proof} This is obtained directly from the unitarity and the QYBE for $R(z,\mu)$.
\end{proof}
The classical $R$-matrix Dunkl operators can now be written as
\begin{equation}\label{RDctrig}
y_{i}^{\cc}=p_i-g\sum_{j(\ne i)}^n \phi(x_i-x_j, \lambda_i-\lambda_j) \, {S}_{ij}\,,\qquad \phi(z,\mu):= \cot\frac{z}{2}+\cot\frac{\mu}{2}\,,
\end{equation}
and we need to substitute these into classical CMS Hamiltonians. Note that, by the lemma, $S_{ij}$ satisfy the same relations as $s_{ij}\in W$. Also, they interact with $x, p$ in exactly the same way, namely,
\begin{equation}
    S_{ij}f(x,p)S_{ij}=s_{ij}f(x,p)s_{ij}=f(s_{ij}.x, s_{ij}. p)\,.
\end{equation}
Hence, calculating $f(\lambda, y^\cc)$ works in the same way as for the scalar trigonometric Dunkl operators (with spectral parameters $\lambda$), 
\begin{equation}
 y_{i}^{\cc}=p_i-g\sum_{j(\ne i)}^n \phi(x_i-x_j, \lambda_i-\lambda_j) \, s_{ij}.    
\end{equation}
In that case we know (by trigonometric limit from the elliptic case) that the result is a classical  CMS Hamiltonian, independent of $\lambda$, i.e., an element of $\c[h_1,\dots, h_n]$. Therefore, the same is true for their $R$-matrix analogues. This concludes the proof. \qed 

\subsection{} All other constructions work in the trigonometric case in the same way as in the elliptic case. For example, the formula \eqref{reCMS} is valid with $\wp(x_i-x_j)$ replaced by $\displaystyle\frac{1}{4 \sin^{2}\frac{x_i-x_j}{2}}$. The freezing procedure in Proposition \ref{mainsp} now involves evaluating at an equilibrium $(x^*, p^*)$ of the CMS Hamiltonian \eqref{cmcs}, given by 
\begin{equation}
  x^*=(x^*_j)\,,\quad p^*=(p^*_j)\,,\qquad x_j^*=\frac{2\pi j}{n}\,,\quad p_j^*=0\,,\quad j=1,\dots, n.
\end{equation}
In particular, the spin chain Hamiltonian \eqref{h2} in the trigonometric case becomes   \begin{equation}\label{h2trig}
\H_2=\sum_{i<j}\bar r_{ij}\,,\qquad \bar r_{ij}:=r_{ij}\left(\frac{2\pi(i-j)}{n}\right)\,,   \end{equation}
where $r(z)$ is the classical trigonometric $r$-matrix determined from the expansion \eqref{expR}.
\remark{Another case when the above constructions fully apply is the trigonometric supersymmetric $R$-matrix given in \cite[(2.7)]{MZ2}. All the required properties of $R$ can be found in {\it loc.~cit.}.}

\end{document}